\documentclass[oneside, 11pt]{article}
\usepackage{setspace}
\usepackage{titlesec}
\usepackage[utf8]{inputenc}
\usepackage{tabularx}	
\usepackage{etoolbox}
\usepackage{mathtools}
\usepackage{amsmath}
\usepackage{amssymb} 
\usepackage{amsfonts}
\usepackage{mathrsfs} 
\usepackage{amsthm,pifont}
\usepackage{graphicx}
\usepackage{bbm}
\usepackage[colorlinks=true, pdfstartview=FitV, linkcolor=blue,
citecolor=blue, urlcolor=blue, pagebackref=false]{hyperref}
\usepackage{cleveref}
\usepackage[nottoc,numbib]{tocbibind}
\usepackage{xcolor}
\usepackage{tikz}
\usepackage{cite}
\usepackage{braket}
\usepackage{leftidx,tensor}
\usepackage{dsfont}
\usepackage{calc}
\usepackage{caption}
\usepackage{subcaption}
\usepackage{color}
\usepackage[printonlyused,nohyperlinks]{acronym}
\usetikzlibrary{decorations.markings}
\usetikzlibrary{shapes.geometric}
\usetikzlibrary{patterns}
\tikzstyle{vertex}=[circle, draw, inner sep=0pt, minimum size=6pt]

\definecolor{superlightgray}{RGB}{235,235,235}
\definecolor{wellgray}{RGB}{170,170,170}
\definecolor{lightergray}{RGB}{225,225,225}

\usepackage[utf8]{inputenc}
\usepackage[english]{babel}
\usepackage{enumitem}

\allowdisplaybreaks

\newtheorem{theorem}{Theorem}[section]
\makeatletter
\patchcmd{\ttlh@hang}{\parindent\z@}{\parindent\z@\leavevmode}{}{}
\patchcmd{\ttlh@hang}{\noindent}{}{}{}
\makeatother
\titleformat*{\section}{\large\bfseries}
\titleformat*{\subsection}{\small\bfseries}
\titleformat*{\subsubsection}{\small\bfseries}
\titleformat*{\paragraph}{\small\bfseries}
\titleformat*{\subparagraph}{\small\bfseries}

\newcommand{\N}{\mathbb{N}}
\newcommand{\R}{\mathbb{R}}
\newcommand{\Z}{\mathbb{Z}}
\newcommand{\E}{\mathbb{E}}
\newcommand{\p}{\mathbb{P}}
\newcommand{\md}{\ensuremath{\mathrm{d}}}
\newcommand{\dxp}{\theta}

\newcommand{\eps}{\varepsilon}
\newcommand{\dia}{\text{Diam}}
\newcommand{\rad}{\text{Rad}}

\newcommand{\mz}{\mathbf{0}}
\newcommand{\cA}{\mathcal{A}}
\newcommand{\cB}{\mathcal{B}}

\newcommand{\cH}{\mathcal{H}}
\newcommand{\cN}{\mathcal{N}}

\newcommand{\cS}{\mathcal{S}}
\newcommand{\cX}{\mathcal{X}}
\newcommand{\cW}{\mathcal{W}}

\newtheorem{lemma}[theorem]{Lemma}
\newtheorem{remark}[theorem]{Remark}
\newtheorem{proposition}[theorem]{Proposition}
\newtheorem{definition}[theorem]{Definition}
\newtheorem{corollary}[theorem]{Corollary}
\newtheorem{claim}[theorem]{Claim}

\usepackage{geometry}
\begin{document}
	
\title{The polynomial growth of the infinite long-range percolation cluster}

\author{Johannes B\"aumler \footnote{E-Mail: \href{mailto:jbaeumler@math.ucla.edu}{jbaeumler@math.ucla.edu} \textsc{Department of Mathematics, University of California, Los Angeles}. } }

\maketitle
	
	\begin{center}
		\parbox{13cm}{ \textbf{Abstract.} We study independent long-range percolation on $\mathbb{Z}^d$ where the nearest-neighbor edges are always open and the probability that two vertices $x,y$ with $\|x-y\|>1$ are connected by an edge is proportional to $\frac{\beta}{\|x-y\|^s}$, where $\beta>0$ and $s> 0$ are parameters. We show that the ball of radius $k$ centered at the origin in the graph metric grows polynomially if and only if $s\geq 2d$. For the critical case $s=2d$, we show that the volume growth exponent is inversely proportional to the distance growth exponent. Furthermore, we provide sharp upper and lower bounds on the probability that the origin and $ne_1$ are connected by a path of length $k$ in the critical case $s=2d$. We use these results to determine the Hausdorff dimension of the critical long-range percolation metric that was recently constructed by Ding, Fan, and Huang [14].}
	\end{center}

\let\thefootnote\relax\footnotetext{{\sl MSC Class}: 05C12, 60K35, 82B27, 82B43}
\let\thefootnote\relax

\hypersetup{linkcolor=black}
\tableofcontents
\hypersetup{linkcolor=blue}

\section{Introduction}

\subsection{Introduction and related work}

Consider independent long-range bond percolation on $\Z^d$ where an edge $\{x,y\}$ is open with probability
\begin{equation*}
	p(\beta,\{x,y\}) \coloneqq 1-e^{-\beta J(x-y) } ,
\end{equation*}
where $J:\Z^d \to \left[0,\infty\right]$ is a kernel and $\beta \geq 0$ is a parameter. Here we define that $0\cdot\infty=\infty$ and $e^{-\infty}=0$. If there is an edge between $x$ and $y$, we write $x\sim y$. In this paper, we only consider the case where
\begin{equation*}
	J(x-y) = +\infty \text{ if } \|x-y\|=1,
\end{equation*}
with $\|\cdot\| = \|\cdot\|_2$ denoting the usual Euclidean norm on $\R^2$.
So in particular, the graph with vertex set $\Z^d$ equipped with the open edges is almost surely connected. We are interested in the size and structure of the ball $\mathcal{B}_k$ of radius $k$ around the origin, which is the set
\begin{equation*}
	\cB_k = \left\{x\in \Z^d : D(\mz,x)\leq k\right\},
\end{equation*}
where
\begin{equation*}
	D(\mz,x) = \inf\{n\in \N: \text{ There is an open path of length $n$ from $\mz$ to $x$}\}
\end{equation*}
is the graph distance (also called chemical distance or hop-count distance) between $\mz$ and $x$. As the nearest-neighbor edges are always connected, it directly follows that $D(x,y)\leq\|x-y\|_1$ for all $x,y\in \Z^d$. For a set $A\subset \Z^d$ and points $x,y\in A$ we define
\begin{equation*}
	D_A(x,y) \coloneqq \inf\{n\in \N: \text{ There is an open path of length $n$ from $x$ to $y$ inside $A$}\}
\end{equation*}
as the distance between $x$ and $y$ within the set $A$, and we define the diameter of a set $A\subset \Z^d$ by
\begin{equation*}
	\dia(A) = \sup_{x,y \in A} D_A(x,y)
\end{equation*}
as the maximum distance between points $x,y\in A$.
The set $\cB_k$ can also be an infinite set, which happens in the case where
\begin{equation}\label{eq:finite degree}
	\E \left[\deg(\mz)\right] = \sum_{x\in \Z^d \setminus \{\mz\}} \p \left(\mz \sim x\right) = \infty ,
\end{equation}
where we write $\deg(y)$ for the degree of a vertex, which is the number of open edges with $y$ as an endpoint. Indeed, if the expected degree of the origin is infinite, a standard Borel-Cantelli argument together with the independence of different edges implies that $|\cB_1|=\infty$ almost surely. On the other hand, if the expected degree of the origin is finite, a comparison with a Galton-Watson tree shows that for all $k\in \N$
\begin{equation*}
	\E \left[ |\cB_k| \right] \leq \E\left[\deg(\mz)\right]^k
\end{equation*}
and thus $\cB_k$ is almost surely a finite set for all $k\in \N$.

In this paper, we are mostly interested in the case where the kernel $J$ decays polynomially in the Euclidean distance, i.e. $J(x) = \Theta\left(\|x\|^{-s}\right)$ for some $s>0$. Depending on the value of $s$, the expected degree can be finite or infinite; summing over all possible $x\neq \mz$ in \eqref{eq:finite degree} shows that the expected degree is finite if and only if $s>d$. Regarding the chemical distances in these random graphs, there are five different regimes for the growth of the chemical distance, depending on the exponent $s$.
\begin{itemize}
	\item For $s<d$, Benjamini, Kesten, Peres, and Schramm \cite{benjamini2011geometry} proved that the graph distance between any two points $x,y\in \Z^d$ is at most $\lceil\frac{d}{d-s}\rceil$, and the graph distance between $x$ and $y$ converges to $\lceil\frac{d}{d-s}\rceil$ in probability as $\|x-y\| \to \infty$.
	
	\item For $s=d$, Coppersmith, Gamarnik, and Sviridenko \cite{coppersmith2002diameter} determined the asymptotic growth of the chemical distance to be of order $\frac{\log(\|x-y\|)}{\log \log(\|x-y\|)}$. More precise asymptotics of the diameter were recently determined by Wu \cite{wu2022sharp}.
	
	\item For $s\in (d,2d)$, Biskup et al. \cite{biskup2004scaling, biskup2011graph, biskup2019sharp} proved that the graph distance between $x$ and $y$ is of order
	\begin{equation*}
		\log(\|x-y\|)^\Delta \ \ \text{ with } \Delta^{-1} = \log_2 \left(\frac{2d}{s}\right).
	\end{equation*}
	Note that the graph containing the open edges is locally finite if and only if $s>d$, so for $s\in (d,2d)$ the long-range percolation graph is locally finite.
	
	\item  For $s=2d$, the author proved that the graph distance between $x$ and $y$ grows polynomially in the Euclidean distance \cite{baumler2022distances}. In dimension $d=1$, this was previously proven by Ding and Sly \cite{ding2013distances}.
	
	\item For $s>2d$, Berger proved that the graph distance between $x$ and $y$ grows linearly in the Euclidean distance \cite{berger2004lower}.
\end{itemize}

Some of the results stated above also hold in the case where $J$ does not take the value $+\infty$, i.e., when there are no edges that are almost surely open. In this case, one restricts to points $x$ and $y$ in the supercritical infinite open cluster. In particular, it follows from the results of \cite{baumler2023continuity} that in dimensions $d\geq 2$ the typical distance between two points $x$ and $y$ in a supercritical infinite long-range percolation cluster with polynomially decaying connections ($s>d$) can not grow faster than the Euclidean distance. For the case where $J$ does not take the value $+\infty$, long-range percolation also can have a percolation phase transition in one dimension, depending on the speed of the decay of $J$ \cite{aizenman1986discontinuity,newman1986one,schulman1983long}. Also the continuity of the phase transition is well-understood  for long-range percolation in one dimension \cite{hutchcroft2021power,hutchcroft2022sharp,berger2002transience,aizenman1986discontinuity}.\\

In this paper, we are interested in the asymptotic growth of the ball $\cB_k$ as $k\to \infty$. This problem was previously considered by Trapman \cite{trapman2010growth}, who determined that $s=d$ is the critical value for the growth of the ball when considered at an exponential scale. Assuming that $\E \left[\deg(\mz)\right]=\infty$, it directly follows that $|\cB_1|=\infty$ and thus $|\cB_k|^{1/k}$ diverges to $+\infty$ as $k\to \infty$. Contrary to that, if $s>d$, Trapman showed that $\cB_k$ grows subexponentially, i.e., $|\cB_k|^{1/k}$ converges to $1$ as $k\to \infty$. For the critical case $s=d$, he proved that there are kernels $J$ such that $J(x)=\|x\|^{-d+o(1)}$ for which there exist constants $1< a_1\leq a_2 < \infty$ such that
\begin{equation*}
	\lim_{k\to \infty} \p \left( a_1^k \leq |\cB_k| \leq a_2^k \right) = 1.
\end{equation*}
In this paper, we are interested not in the exponential growth of $\cB_k$, but in the polynomial growth. Here, we identify $s=2d$ to be the critical exponent for the long-range model. Before stating our results, we first review the setup that was considered in \cite{baumler2022behavior,baumler2022distances,ding2013distances} to study the chemical distances in the case $s=2d$.

\subsection{The self-similar case}\label{sec:self sim}

Assume that different edges are independent of each other and that an edge $\{u,v\}$ is open with probability
\begin{equation*}
	p(\beta,\{u,v\})  = 1-e^{-\beta J(u-v)},
\end{equation*}
where $\beta \geq 0$ is a parameter and
\begin{equation}\label{eq:selfsim kernel}
	J(u-v)=\int_{u+\left[0,1\right)^d} \int_{v+\left[0,1\right)^d} \frac{1}{\|x-y\|^{2d}} \md x \md y, 
\end{equation}
with $\|\cdot\|$ denoting the Euclidean norm. We call the corresponding probability measure $\p_\beta$ and denote its expectation by $\E_\beta$. As the integral $\int_{u+\left[0,1\right)^d} \int_{v+\left[0,1\right)^d} \frac{1}{\|x-y\|^{2d}} \md x \md y $ diverges for $\|u-v\|_\infty = 1$, this implies that $J(u-v)=+\infty$ when $\|u-v\|_\infty = 1$ and thus the nearest-neighbor edges (in the $\infty$-norm) are connected with probability $1$. In this setup, it was proven in \cite[Theorem 1.1]{baumler2022distances} that for all $\beta \geq 0$ and $d$ there exists an exponent $\theta = \theta(\beta,d)$ such that
\begin{align}\label{eq:theo:exponent1}
	\|u\|^\dxp \approx_P D\left(\mz,u\right)  \approx_P \E_\beta\left[D(\mz,u)\right]
\end{align}
and
\begin{align}\label{eq:theo:exponent2}
	n^\dxp \approx_P \dia\left( \left\{0,\ldots, n\right\}^d \right)  \approx_P \E_\beta\left[ \dia\left( \left\{0,\ldots, n\right\}^d \right) \right]\text,
\end{align}
where the notation $A(n)\approx_P B(n)$ means that for all $\eps >0 $ there exist $0<c<C<\infty$ such that $\p\left(cB(n) \leq A(n) \leq CB(n)\right)> 1-\eps$ for all $n \in \N$. 
Regarding the tail behavior of the diameter, it was proven \cite[Theorem 6.1]{baumler2022distances} that
\begin{align}\label{eq:stretched upper bound}
	\sup_{n \in \N } \E_\beta \left[ \exp \left( \left( \frac{\dia \left(\{0,\ldots,n\}^d\right)}{n^{\dxp(\beta,d)}}\right)^{\eta} \right) \right] < \infty  
\end{align}	
for all $\eta < \frac{1}{1-\theta(\beta,d)}$. Furthermore, it is known that the exponent $\theta(\beta,d)$ is continuous and strictly decreasing as a function in $\beta$ and there are some results on the limiting behavior as $\beta \to 0$, respectively $\beta \to \infty$, see \cite{baumler2022behavior} and \cite[Theorem 1.2]{baumler2022distances}. In this paper, we will use that for all dimensions $d\geq 1$
\begin{equation}\label{eq:theta knowledge}
	\theta(\beta,d) \text{ is continuous in $\beta$}, \ \theta(0,d)=1, \text{ and that } \lim_{\beta \to \infty} \theta(\beta,d) = 0 .
\end{equation}
Regarding the particular choice of the kernel $J$ in \eqref{eq:selfsim kernel}, there is an underlying continuous model and a coupling with a Poisson point process that gives rise to exactly this kernel. The precise details of this construction were given in \cite[Section 1.2]{baumler2022distances} and we do not repeat it here. The most important feature of this construction is an invariance under renormalizations in this model. For $u\in \Z^d$ and $n\in \N$, define
\begin{equation*}
	V_u^n \coloneqq n\cdot u + \{0,\ldots,n-1\}^d = n \cdot u + V_\mz^n,
\end{equation*}
which implies that $\left(V_u^n\right)_{u\in \Z^d}$ is a tessellation of $\Z^d$. With the connection probability defined through \eqref{eq:selfsim kernel} one has for all $u,v\in \Z^d$ with $u\neq v$ that
\begin{equation*}
	\p_\beta \left(V_u^n \sim V_v^n\right) = \p_\beta \left( u \sim v \right) .
\end{equation*}
So if one starts with long-range percolation graph with vertex set $\Z^d$ and then identifies boxes of the form $V_u^n$ to one point, call it $r(u)$, then the resulting random graph $G^\prime=(V^\prime,E^\prime)$ with vertex set $V^\prime=\left\{r(u) : u\in \Z^d\right\}$ and edge set $E^\prime =\{\{r(u),r(v)\}: V_u^n \sim V_v^n\}$ has the same distribution as the original long-range percolation graph. This is also why the kernel $J$ defined in \eqref{eq:selfsim kernel} is called the {\sl self-similar} kernel and the resulting model is called the {\sl self-similar} long-range percolation model.

\subsection{Main results}

After reviewing these results for the self-similar model, we can present our main results. Remember that we write $|\cB_k| = \left|\left\{x \in \Z^d : D(\mz,x) \leq k\right\}\right|$ for the size of the ball of radius $k$ around the origin. The first result concerns the probability that $\left|\cB_k\right| k^{-\frac{d}{\theta(\beta,d)}}$ is large. In Corollary \ref{coro1} below, we will see that $k^{\frac{d}{\theta(\beta,d)}}$ is the typical size of the ball $\cB_k$. 
\begin{theorem}\label{theo:main}
	For all $\beta > 0$ and $d\geq 1$, there exists a constant $C < \infty$ such that for all large enough $n,K \in \N$
	\begin{equation*}
		\p_\beta \left( |\cB_n| \geq C K n^{\frac{d}{\theta(\beta,d)}} \right) \leq 1.5^{-K}.
	\end{equation*}
\end{theorem}

With this theorem, we can already determine the asymptotic growth of the ball $\cB_k$ under the measure $\p_\beta$.

\begin{corollary}\label{coro1}
	For all $\beta > 0$ and $d\geq 1$ one has
	\begin{equation*}
		|\cB_k| \approx_P k^{\frac{d}{\dxp(\beta,d)}}  \ \ \text{ and } \ \ \E_\beta\left[|\cB_k| \right] = \Theta \left(  k^{\frac{d}{\dxp(\beta,d)}} \right).
	\end{equation*}
	Furthermore,
	\begin{equation*}
		\lim_{k\to \infty} \frac{\log(|\cB_k|)}{\log(k)} = \frac{d}{\theta(\beta,d)} 
	\end{equation*}
	almost surely under the measure $\p_\beta$.
\end{corollary}

Using a comparison argument between different percolation measures, Corollary \ref{coro1} already implies the following result.

\begin{corollary}\label{coro2}
	Let $\p$ be a measure of independent bond percolation on $\Z^d$ such that $\p \left( x\sim y \right) = 1$ if $\|x-y\|=1$. Define $\beta^{\star}$ and $\beta_{\star}$ by
	\begin{align*}
		\beta^{\star} & \coloneqq \limsup_{\|x-y\| \to \infty} \|x-y\|^{2d} \p(x\sim y) \ \text{ and } \
		\beta_{\star}  \coloneqq \liminf_{\|x-y\| \to \infty} \|x-y\|^{2d} \p(x\sim y).
	\end{align*}
	If $0<\beta_\star \leq \beta^\star < \infty$, then
	\begin{align}
		\label{eq:coro 2 1}& \liminf_{k\to \infty} \frac{\log(|\cB_k|)}{\log(k)} \geq \frac{d}{\theta(\beta_{\star},d)} \ \ \p\text{-almost surely, and } \\
		\label{eq:coro 2 2}& \limsup_{k\to \infty} \frac{\log(|\cB_k|)}{\log(k)} \leq \frac{d}{\theta(\beta^{\star},d)} \ \ \p\text{-almost surely.}
	\end{align}
	In particular, if $\beta_{\star}=\beta^{\star}=\beta \in (0,\infty)$, then 
	\begin{equation}\label{eq:coro 2 3}
		\lim_{k\to \infty} \frac{\log(|\cB_k|)}{\log(k)} = \frac{d}{\theta(\beta,d)} \ \ \p\text{-almost surely}.
	\end{equation}
	If $\p(x\sim y)= o(\|x-y\|^{-2d})$, i.e, if $\beta^{\star}=0$, then 
	\begin{equation}\label{eq:coro 2 4}
		\lim_{k\to \infty} \frac{\log(|\cB_k|)}{\log(k)} = d \ \ \p\text{-almost surely},
	\end{equation}
	and if $\p(x\sim y)= \omega(\|x-y\|^{-2d})$, i.e, if $\beta_{\star}=+\infty$, then 
	\begin{equation}\label{eq:coro 2 5}
		\lim_{k\to \infty} \frac{\log(|\cB_k|)}{\log(k)} = +\infty \ \ \p\text{-almost surely}.
	\end{equation}
\end{corollary}

\noindent
One other problem we consider is the deviation from the median of the graph distance $D(\mz,x)$ under the measure $\p_\beta$. Using the stretched-exponential tightness \eqref{eq:stretched upper bound} and Markov's inequality one gets an upper bound on the probability of events of the form $\left\{D(\mz,x)\geq K \|x\|^\theta\right\}$ for $K$ large. In the next theorem, we are instead interested in the lower tail of the chemical distance, i.e., in the probability of events of the form $\left\{D(\mz,x) \leq \eps \|x\|^\theta\right\}$. Here we prove up-to-constants upper and lower bounds.

\begin{theorem}\label{theo:lower tail}
	For all $d$ and $\beta> 0$, there exist constants $0<c<C<\infty$ such that for all $x\in \Z^d\setminus\{\mz\}$ and all $\eps \in (0,1)$ with $\eps \|x\|^{\theta(\beta,d)} \geq 1$ one has
	\begin{equation}\label{eq:lower tail}
		c \eps^{\frac{2d}{\theta}} \leq
		\p_\beta \left( \frac{D(\mz,x)}{\|x\|^{\theta}} \leq \eps  \right) 
		\leq C \eps^{\frac{2d}{\theta}}.
	\end{equation}
\end{theorem}

Following this theorem, the behavior of the lower tail is better understood than that of the upper tail. Indeed, for the upper tail, it is known that $\frac{D(\mz,x)}{\|x\|^\theta}$ has uniform stretched exponential moments for $\eta < \frac{1}{1-\theta(\beta,d)}$, i.e., that
\begin{align*}
	\sup_{x \in \Z^d \setminus \{\mz\} } \E_\beta \left[ \exp \left( \left( \frac{D(\mz,x)}{\|x\|^{\dxp(\beta,d)}}\right)^{\eta} \right) \right] < \infty  
\end{align*}
for $\eta < \frac{1}{1-\theta(\beta,d)}$ \cite[Theorem 6.1]{baumler2022distances}, and that the corresponding result is not true for $\eta > \frac{d}{1-\theta(\beta,d)}$, cf. \cite[Lemma 6.2]{baumler2022distances}. These bounds are only matching in dimension $d=1$, and for general dimension $d\geq 2$ it is not known so far what happens for $\eta \in \left(\frac{1}{1-\theta(\beta,d)},\frac{d}{1-\theta(\beta,d)}\right)$. For the lower tail, Theorem \ref{theo:lower tail} provides matching upper and lower bounds for all dimensions.

\subsection{Hausdorff dimension of the limiting metric space}

The volume growth plays an important role in the limiting metric space of the self-similar long-range percolation model. The limiting metric of self-similar percolation was recently proven to exist by Ding, Fan, and Huang \cite{ding2023uniqueness}. The idea is that one can define a limiting metric $\hat{D}:\R^d \times \R^d \to \left[0,\infty\right)$ that is a limit of the discrete long-range percolation metrics. For this scaling limit they proved, among many other things, the following result.

\begin{theorem}[Theorem 1.1 in \cite{ding2023uniqueness}]\label{theo:ding}
	Let $D$ be the chemical distance on the discrete $\beta$-long-range percolation model. Let $\widehat{a}_n$ be the median of $D(\mathbf{0}, n \mathbf{1})$ (here $\mathbf{1}=(1,1, \ldots, 1) \in \mathbb{R}^d$) and $\widehat{D}_n=\widehat{a}_n^{-1} D(\lfloor n \cdot\rfloor,\lfloor n \cdot\rfloor)$. Then there exists a unique random metric $\hat{D}$ on $\mathbb{R}^d$ such that $\widehat{D}_n$ converges to $\hat{D}$ in law with respect to the topology of local uniform convergence on $\mathbb{R}^{d} \times \R^d$.
\end{theorem}

Here, for a vector $x\in \R^d$, we define $\lfloor x \rfloor$ as the component-wise floor. One natural question about this metric space is, whether its Hausdorff dimension can be explicitly computed, depending on $\beta$ or $\theta(\beta,d)$. For a metric space $(Y,\tilde{D})$ and a subset $X\subset Y$, the Hausdorff dimension of the metric space $(X,\tilde{D})$ is defined as follows. For $\Delta \geq 0$, define the $\Delta$-Hausdorff content of $(X,\tilde{D})$ as
\begin{align*}
	& C_\Delta \left( X ,\tilde{D} \right) = \inf \left\{\sum_{i=1}^{\infty} r_i^\Delta \Big| \text{ There is a cover of $X$ by $\tilde{D}$-balls of radii $\left(r_i\right)_{i\geq 1}$}  \right\}.
\end{align*}
The Hausdorff dimension of $(X,\tilde{D})$ is now defined by
\begin{equation*}
	\dim_{\cH}^{\tilde{D}}(X) = \inf\left\{ \Delta \geq 0 : C_\Delta\left(X, \tilde{D}\right)=0\right\}.
\end{equation*} 
For $\beta\geq 0$, we also write $\dim_{\cH}^\beta(X)$ for the Hausdorff dimension of a set $X\subseteq \R^d$ equipped with the limiting long-range percolation metric $\hat{D}$ constructed through Theorem \ref{theo:ding}. For the special case $\beta=0$, the long-range percolation metric $\hat{D}$ just equals the metric induced by the $\infty$-norm on $\R^d$. So in particular, for any set $X\subset \R^d$, the Hausdorff dimension $\dim_{\cH}^0(X)$ just equals the usual Hausdorff dimension of $X$ as a subset of $\R^d$.

It was shown in \cite[Theorem 1.16]{ding2023uniqueness} that for any deterministic subset $X\subset \R^d$ one has
\begin{equation}\label{eq:ding hausdorff}
	\dim_{\cH}^\beta (X) \leq \frac{\dim_{\cH}^0 (X)}{\theta(\beta,d)}
\end{equation}
almost surely, and it was conjectured that this inequality is actually an equality, cf. \cite[Remark 1.17]{ding2023uniqueness}. In this paper, we verify this conjecture.

\begin{theorem}\label{theo:hausdorff dimension}
	Let $X$ be a deterministic Borel subset of $\R^d$. Then we have almost surely
	\begin{equation}\label{eq:hausdorff}
		\dim_{\cH}^{\beta} (X) \geq \frac{\dim_{\cH}^0 (X)}{\theta(\beta,d)}.
	\end{equation}
\end{theorem}

Note that Theorem \ref{theo:hausdorff dimension} and \eqref{eq:ding hausdorff} directly imply that
\begin{equation*}
	\dim_{\cH}^{\beta} \left(O\right) = \frac{d}{\theta(\beta,d)}
\end{equation*}
for any open set $O \subset \R^d$ almost surely,
so the Hausdorff dimension of the metric space $(\R^d,\hat{D})$ equals its ``volume-growth exponent''.\\

\textbf{Notation:} 
We typically consider the collection of open/closed edges as an element $\omega \in \{0,1\}^E$. We say that an edge is open if $\omega(e)=1$ and closed otherwise. When we want to measure the graph distance between two points $x$ and $y$ in the percolation environment defined by $\omega$, we also write $D(x,y;\omega)$. The distance $D_A(x,y)$ is defined as the length of the shortest open path between $x$ and $y$ that lies entirely within $A$. If we look at the ball of radius $k$ around the origin in the environment $\omega$, we also write
\begin{equation*}
	\cB_k(\omega) = \left\{x \in \Z^d : D(\mz,x;\omega)\leq k\right\} .
\end{equation*}
For a metric $\tilde{D}: \R^d \times \R^d \to \left[0,\infty\right)$, we write $B_r(x;\tilde{D}) = \left\{ y \in \R^d : \tilde{D}(x,y)\leq r\right\}$ for the $\tilde{D}$-ball of radius $r$ around $x$. If we do not specify a metric, we refer to the usual Euclidean metric.
For two sets $A,B \subset \Z^d$, we define the indirect distance $D^\star (A,B)$ between $A$ and $B$ as the length of the shortest open path between $A$ and $B$ that does not use a direct edge between $A$ and $B$. For a vertex $u\in \Z^d$ we write
\begin{equation*}
	\deg(u) = \sum_{v\in \Z^d \setminus \{u\}} \mathbbm{1}_{u\sim v}
\end{equation*}
for the degree of $u$. For a vertex $r(u) \in V^\prime$, i.e., an element of the renormalized graph constructed in Section \ref{sec:self sim}, we write
\begin{equation*}
	\deg(r(u)) = \sum_{r(v) \in V^\prime \setminus \{r(u)\}} \mathbbm{1}_{r(u)\sim r(v)}
\end{equation*}
for its degree. We define the {\sl neighborhood} of $r(u)$ by $\cN\left(r(u)\right) = \left\{r(v) \in V^\prime : \|u-v\|_\infty \leq 1\right\}$, and we define the {\sl neighborhood-degree} of $r(u)$ by
\begin{equation}\label{eq:nbh degree}
	\deg^{\cN}(r(u)) = \sum_{r(v) \in \cN(r(u))} \deg(r(v)).
\end{equation}
For a finite set $Z\subset \Z^d$, we define its average degree by
\begin{equation*}
	\overline{\deg}(Z) = \frac{1}{|Z|} \sum_{u\in Z} \deg(u).
\end{equation*}

\section{Proof of the Corollaries}\label{sec:coro proofs}

Let us first see how the corollaries are implied by Theorem \ref{theo:main}. We start with the proof of Corollary \ref{coro1}.

\begin{proof}[Proof of Corollary \ref{coro1} assuming Theorem \ref{theo:main}]
	By Theorem \ref{theo:main} there exists a constant $C< \infty$ such that for all large enough $k,K\in \N$
	\begin{equation}\label{eq:expo decay}
		\p_\beta \left(|\cB_k| \geq C K k^{\frac{d}{\theta(\beta,d)}}\right) \leq 1.5^{-K}. 
	\end{equation}
	Furthermore, we get by Markov's inequality that for fixed $\eps > 0$ one has for all $k \in \N$ with $\eps^{1/d}k^{1/\theta} \geq 1$ that
	\begin{align}
		\notag & \p_\beta \left( |\cB_k| \leq \eps k^{\frac{d}{\theta(\beta,d)}} \right)
		\leq
		\p_\beta \left( \dia \left(\left\{0,\ldots,\lceil \eps^{1/d} k^{1/\theta} \rceil \right\}^d\right) \geq k \right)\\
		\notag &
		=
		\p_\beta \left( \frac{\dia \left(\left\{0,\ldots,\lceil \eps^{1/d} k^{1/\theta} \rceil \right\}^d\right)}{\left(\eps^{1/d} k^{1/\theta}\right)^\theta} \geq \eps^{-\frac{\theta}{d}} \right)
		\\
		&
		\leq \E_\beta \left[\exp\left(\frac{\dia \left(\left\{0,\ldots,\lceil \eps^{1/d} k^{1/\theta} \rceil \right\}^d\right)}{\left(\eps^{1/d} k^{1/\theta}\right)^\theta}\right)\right] \exp \left(-\eps^{-\frac{\theta}{d}}\right) 
		\leq \bar{C} \exp \left(-\eps^{-\frac{\theta}{d}}\right)  \label{eq:markov1}
	\end{align}
	where the constant $\bar{C}$ depends only on the dimension $d$ and $\beta$, but not on $k$ or $\eps$. Such a constant exists by \eqref{eq:stretched upper bound}. So combining \eqref{eq:markov1} and \eqref{eq:expo decay} gives that
	\begin{equation*}
		|\cB_k| \approx_P k^{\frac{d}{\dxp(\beta,d)}}  .
	\end{equation*}
	This also implies that
	\begin{equation}\label{eq:Omega}
		\E_\beta\left[|\cB_k| \right] = \Omega \left(  k^{\frac{d}{\dxp(\beta,d)}} \right).
	\end{equation}
	To get the reverse inequality, note that for large enough $k \in \N$ one has
	\begin{align}\label{eq:O}
		&\notag \E_\beta\left[|\cB_k| \right] = 
		\sum_{j=1}^{\infty} \p_\beta \left( |\cB_k| \geq j \right)
		\leq
		\sum_{i=0}^{\infty} Ck^{\frac{d}{\theta(\beta,d)}} \p_\beta \left( |\cB_k| \geq C i k^{\frac{d}{\theta(\beta,d)}}  \right)\\
		& 
		\leq
		Ck^{\frac{d}{\theta(\beta,d)}}
		\left(K^\prime + \sum_{i=K^\prime}^{\infty}  1.5^{-i}\right)
		= \mathcal{O} \left(  k^{\frac{d}{\dxp(\beta,d)}} \right)
	\end{align}
	for some fixed $K^\prime \in \N$. So \eqref{eq:Omega} and \eqref{eq:O} show that $\E_\beta\left[|\cB_k| \right] = \Theta \left(k^{\frac{d}{\theta(\beta,d)}}\right)$.\\

	Inequality \eqref{eq:markov1} is stated only for $\eps^{1/d}k^{1/\theta}\geq 1$, or equivalently for $\eps k^{d/\theta} \geq 1$. However, it also holds for $\eps k^{d/\theta} < 1$, as the event $\{|\cB_k| < 1\}$ has probability $0$. Thus we get that there exists a constant $\bar{C} = \bar{C}(d,\beta) <\infty$ such that
	\begin{equation}\label{eq:markov2}
		\p_\beta \left( |\cB_k| \leq \eps k^{\frac{d}{\theta(\beta,d)}} \right) \leq 
		\bar{C} \exp \left(-\eps^{-\frac{\theta}{d}}\right) 
	\end{equation}
	for all $\eps > 0$. So for fixed $\delta > 0$, we get that
	\begin{equation}\label{eq:conco1}
		\p_\beta \left( \frac{\log(|\cB_k|)}{\log(k)} \leq \frac{d}{\theta} - \delta \right)
		=
		\p_\beta \left( |\cB_k| \leq k^{-\delta} k^{\frac{d}{\theta}} \right) 
		\overset{\eqref{eq:markov2}}{\leq} \bar{C} \exp \left(- k^{\delta \frac{\theta}{d}} \right).
	\end{equation}
	On the other hand, inequality \eqref{eq:expo decay} implies that
	\begin{equation}\label{eq:conco2}
		\p_\beta \left( \frac{\log(|\cB_k|)}{\log(k)} \geq \frac{d}{\theta} + \delta \right)
		=
		\p_\beta \left( |\cB_k| \geq k^{\delta} k^{\frac{d}{\theta}} \right) 
		\overset{\eqref{eq:expo decay}}{\leq} \exp\left(- c k^{\delta} \right),
	\end{equation}
	where $c>0$ is a small (but positive) constant. A standard Borel-Cantelli argument now implies that $\frac{\log(|\cB_k|)}{\log(k)} \in \left(\frac{d}{\theta} - \delta, \frac{d}{\theta} + \delta\right)$ for all but finitely many $k$ almost surely. Taking $\delta \to 0$ now implies that $\frac{\log(|\cB_k|)}{\log(k)}$ converges to $\frac{d}{\theta}$ almost surely.
\end{proof}

\begin{remark}
	Using the concentration inequalities of \eqref{eq:conco1} and \eqref{eq:conco2} one can also show that $\frac{\log(|\cB_k|)}{\log(k)}$ converges to $\frac{d}{\theta}$ in $L^p$ for all $p\in \left[1,\infty\right)$.
\end{remark}

Next, let us go to the proof of Corollary \ref{coro2}. For the proof of this, we use the {\sl Harris-coupling} for percolation (see \cite[Section 1.3]{heydenreich2017progress}). In this coupling, we start with a collection of independent and identically distributed random variables $\left(U_e\right)_{e\in E}$ that are uniformly distributed on the interval $\left(0,1\right)$. We now define two environments $\omega,\omega_\beta \in \{0,1\}^E$ by
\begin{equation}\label{eq:harris coupling}
	\omega(e) = \mathbbm{1}_{U_e \leq \p(e \text{ open})} \ \ \text{ and } \ \ 
	\omega_\beta(e) = \mathbbm{1}_{U_e \leq \p_\beta(e \text{ open})} .
\end{equation}
This defines a coupling between the measures $\p$ and $\p_\beta$. The important feature of this coupling is its monotonicity. If $\p(e \text{ open}) \leq \p_\beta(e \text{ open})$, and the edge $e$ is open in the environment $\omega$, then $e$ is also open in the environment $\omega_\beta$. With this, we are ready to go to the proof of Corollary \ref{coro2}.

\begin{proof}[Proof of Corollary \ref{coro2} assuming Corollary \ref{coro1}]
	Assume that $$\beta_{\star} = \liminf_{\|x-y\| \to \infty} \|x-y\|^{2d} \p(x\sim y) \in (0,\infty),$$ and let $\eps \in (0,\beta_{\star})$ be arbitrary. Then
	\begin{equation}\label{eq:lower bound 1}
		\p(x\sim y) \geq \frac{\beta_{\star}-\eps/2}{\|x-y\|^{2d}}
	\end{equation}
	for all $x,y$ with $\|x-y\|$ large enough. On the other hand, the kernel $J$ in the definition of $\p_\beta$ satisfies
	\begin{equation*}
		J(x-y) = \int_{x+\left[0,1\right)^d} \int_{y+\left[0,1\right)^d} \frac{1}{\|t-s\|^{2d}} \md t \md s
		=\frac{1}{\|x-y\|^{2d}} + \mathcal{O} \left(\frac{1}{\|x-y\|^{2d+1}}\right),
	\end{equation*}
	as shown in \cite[Example 7.2]{baumler2022distances}. Using that
	\begin{equation*}
		\p_\beta \left(\{x,y\} \text{ open}\right) = 1-e^{-\beta J(x-y)} = \beta J(x-y) + \mathcal{O}(J(x-y)^2) = \frac{\beta}{\|x-y\|^{2d}} + \mathcal{O} \left(\frac{1}{\|x-y\|^{2d+1}}\right)
	\end{equation*}
	we thus get that
	\begin{equation}\label{eq:lower bound 2}
		\p_{\beta_{\star} - \eps} \left(\{x,y\} \text{ open}\right) \leq \frac{\beta_{\star}-\eps/2}{\|x-y\|^{2d}}
	\end{equation}
	for all $x,y$ with $\|x-y\|$ large enough. Say that \eqref{eq:lower bound 1} and \eqref{eq:lower bound 2} hold for all $x,y\in \Z^d$ with $\|x-y\|_1 \geq M$. Now construct the Harris-coupling between the random environments $\omega_{\beta_{\star}-\eps}$ and $\omega$ as described in \eqref{eq:harris coupling}. Remember that the environment $\omega_{\beta_{\star} - \eps}$ is distributed according to the measure $\p_{\beta_{\star} - \eps}$, whereas the environment $\omega$ is distributed according to the measure $\p$. This implies that if $\{x,y\}$ is an edge with $\|x-y\|_1 \geq M$ that is $\omega_{\beta_{\star}-\eps}$-open, then it is also $\omega$-open. If $x,y \in \Z^d$ are such that $\|x-y\|_1 < M$, then the graph distance between $x$ and $y$ in the environment $\omega$ is at most $M$ almost surely, as the nearest-neighbor edges are open in the environment $\omega$ with probability $1$. Thus we get that for any two points $x,y \in \Z^d$ with $\omega_{\beta_{\star}-\eps}(\{x,y\}) = 1$ that $D(x,y;\omega)\leq M$. Using this over shortest paths in the environment $\omega_{\beta_{\star}-\eps}$, one gets that for all $x,y \in \Z^d$
	\begin{equation*}
		D(x,y;\omega) \leq M D(x,y;\omega_{\beta_{\star}-\eps})
	\end{equation*}
	which implies that for all $k\in \N$
	\begin{equation*}
		\cB_{M k}(\omega) \coloneqq  \left\{x \in \Z^d : D(\mz,x; \omega) \leq M k \right\} 
		\supseteq 
		\left\{x \in \Z^d : D(\mz,x; \omega_{\beta_{\star}-\eps}) \leq k \right\} 
		\eqqcolon
		\cB_{k}(\omega_{\beta_{\star}-\eps}).
	\end{equation*}
	As the sequence of sets $\left(\cB_{k}(\omega_{\beta_{\star}-\eps})\right)_{k\in \N_0}$ is distributed like the sequence of sets $\left(\cB_k\right)_{k\in \N_0}$ under the measure $\p_{\beta_{\star}-\eps}$, we get that
	\begin{align*}
		\lim_{k \to \infty} \frac{\log(|\cB_{k}(\omega_{\beta_{\star}-\eps})|)}{\log(k)} = \frac{d}{\theta(\beta_{\star}-\eps,d)}
	\end{align*}
	almost surely, by Corollary \ref{coro1}. Using that $|\cB_{M k}(\omega)| \geq |\cB_{k}(\omega_{\beta_{\star}-\eps})|$, we get that
	\begin{align*}
		\liminf_{k \to \infty} \frac{\log(|\cB_{k}(\omega)|)}{\log(k)} &
		=
		\liminf_{k \to \infty} \frac{\log(|\cB_{Mk}(\omega)|)}{\log(Mk)}
		=
		\liminf_{k \to \infty} \frac{\log(|\cB_{Mk}(\omega)|)}{\log(k)}\\
		&
		\geq
		\liminf_{k \to \infty} \frac{\log(|\cB_{k}(\omega_{\beta_{\star} - \eps})|)}{\log(k)}
		= \frac{d}{\theta(\beta_{\star}-\eps,d)}
	\end{align*}
	almost surely. As $\eps \in (0,\beta_{\star})$ was arbitrary, we get that almost surely
	\begin{equation*}
		\liminf_{k \to \infty} \frac{\log(|\cB_{k}(\omega)|)}{\log(k)} \geq \lim_{\eps \searrow 0} \frac{d}{\theta(\beta_{\star}-\eps,d)}
		=
		\frac{d}{\theta(\beta_{\star},d)}.
	\end{equation*}
	In the last line, we used the continuity of the function $\beta \mapsto \theta(\beta,d)$, see \eqref{eq:theta knowledge}. This shows \eqref{eq:coro 2 1}.
	
	The proof of \eqref{eq:coro 2 2} works similarly, with the roles $\omega_{\beta^\star + \eps}$ and $\omega$ reversed. We omit this proof. Equations \eqref{eq:coro 2 1} and \eqref{eq:coro 2 2} together imply \eqref{eq:coro 2 3}. The results of \eqref{eq:coro 2 4} and \eqref{eq:coro 2 5} follow using a similar argument combined with $\theta(0,d)=1$, respectively $\lim_{\beta \to \infty} \theta(\beta,d) = 0$.
\end{proof}

\section{Proof of Theorems \ref{theo:main}, \ref{theo:lower tail} and \ref{theo:hausdorff dimension}}

\subsection{Previous results}\label{sec:previous}

Before going to the main results, we first need to introduce a few results that were proven in \cite{baumler2022behavior} and \cite{baumler2022distances}. 
The first two results are stated in \cite[Lemmas 5.3 and 5.4]{baumler2022behavior}, see also \cite[Lemma 5.4]{baumler2022distances}. 

\begin{lemma}\label{lem:goodeventB}
	Let $u\in \Z^d$ with $\|u\|_\infty\geq 2$. Let $\mathcal{D}_u (\delta)$ be the following event:
	\begin{align*}
		\mathcal{D}_u (\delta) = \bigcap_{ \substack{ x,y \in V_\mz^{k} : \\ x,y \sim V_u^{k} , x \neq y} } \left\{D_{V_\mz^{k}}\left(x,y\right) \geq \delta k^{\dxp(\beta,d)} \right\} \text .
	\end{align*}
	For every $\beta >0$, there exists a function $f_1(\delta)$ with $f_1(\delta) \underset{\delta\to 0}{\longrightarrow} 1$ such that for all large enough $k \geq k(\delta)$, and all $u \in \Z^d$ with $\|u\|_\infty \geq 2$
	\begin{equation*}
		\p_{\beta} \left(\mathcal{D}_u (\delta) \ | \ V_\mz^{k} \sim V_u^{k} \right) \geq f_1(\delta) \text .
	\end{equation*}
\end{lemma}

\begin{lemma}\label{lem:goodeventA}
	Let $u,v\in \Z^d$ with $u\neq v \neq \mz$ and $\|u\|_\infty \geq 2$. Let $\mathcal{A}_{u,v}(\delta)$ be the following event:
	\begin{align*}
		\mathcal{A}_{u,v} (\delta) 
		= 
		\bigcap_{ \substack{x \in V_\mz^{k}: \\ x \sim V_u^{k}} } \ \bigcap_{ \substack{ y \in V_\mz^{k}: \\ y \sim V_v^{k}} } \left\{D_{V_\mz^{k}}\left(x,y\right) \geq \delta k^{\dxp(\beta,d)} \right\} \text .
	\end{align*}
	For every $\beta >0$, there exists a function $f_2(\delta)$ with $f_2(\delta) \underset{\delta\to 0}{\longrightarrow}1$ such that for all large enough $k \geq k(\delta)$, and all $u,v \in \Z^d\setminus \{\mz\}$ with $\|u\|_\infty \geq 2$
	\begin{equation*}
		\p_{\beta} \left(\mathcal{A}_{u,v}(\delta) \ | \ V_u^{k} \sim V_\mz^{k} \sim V_v^{k} \right) \geq f_2(\delta) \text .
	\end{equation*}
\end{lemma}

The next result was stated in \cite[Lemma 5.5]{baumler2022behavior}, respectively \cite[Lemma 4.6]{baumler2022distances}.

For the next result, we consider conditional probabilities, conditioned on the graph $G^\prime = (V^\prime, E^\prime)$ that was constructed in Section \ref{sec:self sim}. Also remember the definition of the neighborhood-degree in \eqref{eq:nbh degree}.
\begin{lemma}\label{lem:goodeventD}
	Let $\mathcal{D}(\delta)$ be the event
	\begin{align*}
		\mathcal{D}(\delta) \coloneqq \left\{ D^\star \left( V_\mz^{k} , \bigcup_{u \in \Z^d : \| u \|_\infty \geq 2} V_u^{k} \right) \geq \delta k^{\dxp(\beta,d)} \right\} \text .
	\end{align*}
	For every $\beta >0$ there exists a function $f_3(\delta)$ with $f_3(\delta) \underset{\delta\to 0}{\longrightarrow}1$ such that for all large enough $k \geq k(\delta)$, and all realizations of $G^\prime$ with $\deg^{\cN} \left(r(\mz)\right) \leq 9^d 100 \mu_\beta$
	\begin{equation*}
		\p_{\beta} \left( \mathcal{D}(\delta) \ | \ G^\prime \right) \geq f_3(\delta) \text .
	\end{equation*}
\end{lemma}

The next result concerns connected sets in the long-range percolation graph. It was proven in \cite[Lemma 3.2]{baumler2022distances} and \cite[Equation (24)]{baumler2022distances}.

\begin{lemma}\label{lem:connnectedsetsinLRP}
	Let $\mathcal{CS}_k= \mathcal{CS}_k\left(\Z^d\right)$ be all connected subsets of the long-range percolation graph with vertex set $\Z^d$, which are of size $k$ and contain the origin $\mz$. We write $\mu_\beta$ for $\E_\beta \left[\deg(\mz)\right]$. Then for all $\beta>0$ 
	\begin{equation}\label{eq:connected set count bound}
		\E_\beta \left[ |\mathcal{CS}_k| \right] \leq 4^k \mu_\beta^k
	\end{equation}
	and
	\begin{align*}
		\p_\beta \left(\exists Z \in \mathcal{CS}_k : \overline{\deg}(Z) \geq 20 \mu_\beta\right) \leq e^{-4k\mu_\beta}.
	\end{align*}
\end{lemma}

Next, we define a notion of good sets inside the graph with vertex set $\Z^d$. Again, let $G^\prime=(V^\prime,E^\prime)$ be the graph constructed in Section \ref{sec:self sim}. Remember the definition of the events $\mathcal{D}_u(\delta)$,  $\mathcal{A}_{u,v}(\delta)$, and $\mathcal{D}(\delta)$ from Lemmas \ref{lem:goodeventB}, \ref{lem:goodeventA}, and \ref{lem:goodeventD}.

\begin{definition}
	For a small $\delta>0$, we call a vertex $r(w)$ and the underlying block $V_w^{k}$ $\delta$-{\sl good}, if all the translated events of $\mathcal{D}(\delta), \mathcal{D}_u(\delta)$, and $\mathcal{A}_{u,v}(\delta)$ occur, i.e., if
	\begin{align}\label{eq:good1}
		\bigcap_{\substack{ x \in V_w^{k} : \\ x \sim V_u^{k}} } \ \bigcap_{\substack{ y \in V_w^{k} : \\ y \sim V_v^{k}   }  } \left\{D_{V_w^{k}}\left(x,y\right) \geq \delta k^{\dxp(\beta,d)} \right\}
	\end{align}
	for all $u\neq v$ for which $w \neq v$, $V_u^{k} \sim V_w^{k} \sim V_v^{k}$ and $\|u-w\|_\infty \geq 2$, and if
	\begin{align}\label{eq:good2}
		\bigcap_{\substack{ x,y \in V_w^{k} : \\ x ,y \sim V_u^{k} , x\neq y } } 
		\left\{D_{V_w^{k}} (x,y) \geq \delta k^{\dxp(\beta,d)} \right\}
	\end{align}
	for all $u$ with $\|u-w\|_\infty \geq 2$ and $V_u^{k}\sim V_w^{k}$, and if
	\begin{align}\label{eq:good3}
		D^\star \left( V_w^{k} , \bigcup_{u \in \Z^d : \| u - w \|_\infty \geq 2} V_u^{k} \right) \geq \delta k^{\dxp(\beta,d)} .
	\end{align} 
\end{definition}

Suppose that a path crosses a $\delta$-good set $V_w^{k}$, in the sense that it starts somewhere outside of the set $\bigcup_{u \in \Z^d : \| u - w \|_\infty \leq 1} V_u^{k} $, then goes to the set $V_w^{k}$, and then leaves the set $\bigcup_{u \in \Z^d : \| u - w \|_\infty \leq 1} V_u^{k} $ again. When the path enters the set $V_w^{k}$ at the vertex $x$, coming from some a block $V_u^{k}$ with $\|u-w\|\geq 2$, the path needs to walk a distance of at least $\delta k^{\dxp(\beta,d)}$ to reach a vertex $y\in V_w^{k}$ that is connected to the complement of $V_w^{k}$, because of \eqref{eq:good1} and \eqref{eq:good2}. When the path enters the set $V_w^{k}$ from a block $V_v^{k}$ with $\|v-w\|_\infty=1$, then the path crosses the annulus between $V_w^{k}$ and $\bigcup_{u \in \Z^d : \| u - w \|_\infty \geq 2} V_u^{k}$. So in particular, it needs to walk a distance of at least $\delta k^{\dxp(\beta,d)}$ to cross this annulus, because of \eqref{eq:good3}. Overall, we see that the path needs to walk a distance of at least $\delta k^{\dxp(\beta,d)}$ within the set $\bigcup_{u \in \Z^d : \| u - w \|_\infty \leq 1} V_u^{k}$ in order to cross the set $V_w^{k}$.
Define $f(\delta) \coloneqq \min \left\{f_1(\delta) , f_2(\delta) , f_3(\delta)\right\}$.
Let $\delta$ be small enough so that
\begin{align}\label{eq:deltasmall}
	& 9^{2d} 5000 \mu_{\beta+1}^2 \left(1-f(\delta)\right) \leq \left(32\mu_{\beta + 1}\right)^{-9^d400\mu_{\beta+1}}, \text{ and }\\
	& \label{eq:deltasmall2}\left(9^{2d} 5000 \mu_{\beta}^2 \left(1-f(\delta)\right)\right)^{\frac{1}{27^d200\mu_{\beta}}} \leq e^{-5\mu_\beta }.
\end{align}
Such a $\delta >0$ exists, as $f(\delta)$ tends to $1$ for $\delta \to 0$. We fix such a $\delta>0$ for the rest of the paper. The fact that we consider $\mu_{\beta+1}$ in \eqref{eq:deltasmall} is not relevant for this paper but was needed in \cite{baumler2022behavior}, so we stick to this formulation here. Condition \eqref{eq:deltasmall2} was not needed in \cite{baumler2022behavior} but is needed in Lemma \ref{lem:good paths} below.
\begin{definition}
	From here on we call a block $V_w^{k}$ {\sl good} if it is $\delta$-good for the specific choice of $\delta$ in \eqref{eq:deltasmall} and \eqref{eq:deltasmall2}, and we call a vertex $r(w)\in G^\prime$ {\sl good} if the underlying block $V_w^{k}$ is good.
\end{definition}
For a set $Z \subset G^\prime$, we are interested in the number of good vertices inside this set. In \cite[Lemma 5.11]{baumler2022behavior} it was proven that each connected set contains a linear number (in the size of the set) of good vertices with high probability. When there is a linear number of good vertices, then there are also at least two good vertices. Furthermore, it was proven that there is a linear number of {\sl separated good} vertices, which are a special collection of good vertices. We do not need the property of separation in this paper, so we will not go further into this. So the statement of \cite[Lemma 5.11]{baumler2022behavior} directly implies the following result.

\begin{lemma}\label{lem:good sets}
	Let $G^\prime$ be the graph that results from contracting boxes of the form $V_w^{k}$. Then, for large enough $K$, one has
	\begin{align}\label{eq:connectedset goodbound}
		\p_{\beta} \left( \exists Z \in \mathcal{CS}_K\left(G^\prime\right) \text{with less than $2$ good vertices}  \right) \leq 3 \cdot 2^{-K}.
	\end{align}
\end{lemma}

We need to introduce and prove one technical statement regarding a quantity related to the isoperimetric profile of the percolation cluster. For a set $Z\subset \Z^d$, we define $\cS_1(Z) = \left\{x \in \Z^d\setminus Z : x\sim Z\right\}$ as the set of points outside of $Z$ that are connected to $Z$ by an edge. Using this notation, the following result holds.

\begin{lemma}\label{lem:isoper}
	Define the event $\cW_K$ by
	\begin{equation*}
		\cW_K^c = \bigcup_{k=1}^{K} \left\{ \text{There exists $Z \in \mathcal{CS}_k$ s.t. $|\cS_1(Z)| \geq 99 \mu_\beta K$} \right\}.
	\end{equation*}
	Then for all large enough $K$ one has
	\begin{equation*}
		\p_\beta \left(\cW_K^c\right) \leq e^{-50 \mu_\beta K}.
	\end{equation*}
\end{lemma}

Before going to the proof, we introduce the following elementary inequality. Let $\left(U_{i}\right)_{i\in \N}$ be independent Bernoulli random variables and let $U= \sum_{i=1}^{\infty} U_i$ be their sum. Then
\begin{align}\label{eq:exp moment bernoulli}
	\E\left[e^{U}\right] & = \E\left[e^{\sum_{i\in \N} U_i}\right] = \prod_{i\in \N} \E\left[e^{U_i}\right]  \leq \prod_{i\in \N} \left(1+e\E\left[U_i\right]\right) \leq \prod_{i\in \N} e^{e\E\left[U_i\right]} = e^{e\E\left[U\right]} \ 
\end{align}
and this already implies, by Markov's inequality, that for any $C>0$
\begin{align}\label{eq:exp moment bernoulli prob bound}
	\p\left( U > C  \right) = \p\left( e^U > e^{C} \right) \leq \E \left[ e^U\right] e^{-C} \overset{\eqref{eq:exp moment bernoulli}}{\leq} e^{e\E\left[U\right]-C} .
\end{align}

\begin{proof}[Proof of Lemma \ref{lem:isoper}]
	By a union bound over $k=1,\ldots,K$ one gets 
	\begin{equation}\label{eq:union bound k}
		\p_\beta \left(\cW_K^c\right) \leq \sum_{k=1}^{K} \p_\beta \left(\text{There exists $Z \in \mathcal{CS}_k$ s.t. $|\cS_1(Z)| \geq 99 \mu_\beta K$} \right).
	\end{equation}
	Using a further union bound over all possible sets $Z\subset \Z^d$ of size $k$ containing the origin one gets
	\begin{equation*}
		\p_\beta \left(\text{$\exists Z \in \mathcal{CS}_k$ s.t. $|\cS_1(Z)| \geq 99 \mu_\beta K$} \right)
		\leq 
		\sum_{Z\ni \mz: |Z|=k} \p_\beta \left(Z \text{ connected and } |\cS_1(Z)| \geq 99 \mu_\beta K\right) .
	\end{equation*}
	The important thing to notice now is that for fixed a fixed set $Z\subset \Z^d$ the events that $Z$ is connected and that $|\cS_1(Z)| \geq 99 \mu_\beta K$ are independent. Indeed, whether $Z$ is connected depends only on edges with both ends in $Z$, whereas the event that $|\cS_1(Z)| \geq 99 \mu_\beta K$ depends only on edges with at least one endpoint outside of $Z$. For a fixed set $Z$, the size of the set $|\cS_1(Z)|$ is dominated by
	\begin{equation*}
		|\cS_1(Z)| \leq \sum_{z\in Z} \sum_{u \notin Z} \mathbbm{1}_{z \sim u} \eqqcolon Y_Z
	\end{equation*}
	and the random variables $\left(\mathbbm{1}_{z \sim u}\right)_{z\in Z, u \notin Z}$ are independent Bernoulli-distributed random variables. Furthermore $\E \left[Y_Z\right] \leq |Z|\E_\beta\left[\deg(\mz)\right]$.	
	Combining these observations we get that
	\begin{align*}
		& \p_\beta \left(\text{$\exists Z \in \mathcal{CS}_k$ s.t. $|\cS_1(Z)| \geq 99 \mu_\beta K$} \right)
		\leq 
		\sum_{Z\ni \mz: |Z|=k} \p_\beta \left(Z \text{ connected} \right) \p_\beta \left( |\cS_1(Z)| \geq 99 \mu_\beta K\right)\\
		&
		\leq 
		\sum_{Z\ni \mz: |Z|=k} \p_\beta \left(Z \text{ connected} \right) \p_\beta \left( Y_Z \geq 99 \mu_\beta K\right)\\
		&
		\overset{\eqref{eq:exp moment bernoulli prob bound}}{\leq} 
		\sum_{Z\ni \mz: |Z|=k} \p_\beta \left(Z \text{ connected} \right) e^{e\E_\beta \left[Y_Z\right]-99 \mu_\beta K}\\
		&
		\leq
		e^{eK \mu_\beta -99 \mu_\beta K}
		\sum_{Z\ni \mz: |Z|=k} \p_\beta \left(Z \text{ connected} \right)
		\overset{\eqref{eq:connected set count bound}}{\leq} 
		e^{eK \mu_\beta -99 \mu_\beta K}
		4^K \mu_\beta^K \leq e^{-75 \mu_\beta K}.
	\end{align*}
	Plugging this into \eqref{eq:union bound k} implies the result.
\end{proof}

\subsection{The main proposition}

In this section, we present and prove a proposition that is extremely useful in the proofs of Theorems \ref{theo:main} and \ref{theo:hausdorff dimension}. In this proposition we count how many boxes of the form $V_u^k$ the set $B_r(V_\mz^k; D)$ typically intersects. 

\begin{proposition}\label{propo:main}
	For $k\in \N$, we define the {\sl $k$-box-count} of the set $B_r(V_\mz^k; D)$ by
	\begin{equation*}
		X_k \left(B_r(V_\mz^k; D)\right) \coloneqq \left| \left\{u \in \Z^d : V_u^k \cap B_r(V_\mz^k; D) \neq \emptyset\right\} \right|
		=
		\left| \left\{u \in \Z^d : D(V_u^k, V_\mz^k)\leq r\right\} \right|.
	\end{equation*}
	Then for all $\beta > 0$ one has for the choice of $\delta$ in \eqref{eq:deltasmall},\eqref{eq:deltasmall2} that for all large enough $k,K$
	\begin{equation*}
		\p_\beta \left( X_k \left(B_{\delta k^\theta}(V_\mz^k; D)\right) > 3^d 100 \mu_\beta K \right) \leq 4 \cdot 2^{-K}.
	\end{equation*}
\end{proposition}
\begin{proof}
	Define $G^\prime$ as the graph where we renormalized all boxes of the form $V_w^k$. Assume that
	\begin{align*}
		\cX_K \coloneqq \left\{\text{All sets $Z\in \mathcal{CS}_K(G^\prime)$ have at least $2$ good vertices}\right\} \cap \cW_K
	\end{align*}
	holds. Lemmas \ref{lem:good sets} and \ref{lem:isoper} directly imply that $\p_\beta \left( \cX_K^c \right) \leq 4 \cdot 2^{-K}$ for all large enough $K$. Thus, we are left to show that on the event $\cX_K$ the inequality $ X_k \left(B_{\delta k^\theta}(V_\mz^k; D)\right) \leq 3^d 100 \mu_\beta K$ holds.
	Assume that the event $\cX_K$ occurs.
	Define the set $A=\{r(v)\in G^\prime: r(v) \text{ is bad}\}$ and the set $\overline{A}=\left\{r(u)\in G^\prime : r(u) \overset{A\cup\{r(\mz)\}}{\longleftrightarrow} r(\mz) \right\}$. By the definition of the event $\cX_K$, we have $|\overline{A}|\leq K$. Further, we define the sets $A^\prime$ and $A^\star$ by
	\begin{align*}
		&A^\prime = \left\{r(v) \in G^\prime : r(v)\sim \overline{A}\right\} \cup \overline{A} \ \ \text{ and } \ \  A^\star = \bigcup_{r(v) \in A^\prime} \cN(r(v)).
	\end{align*}
	By the definition of the event $\cX_K$, respectively the definition of the event $\cW_K$, we get that $|A^\prime| \leq 99 \mu_\beta K + |\overline{A}| \leq 100 \mu_\beta K$ and thus $|A^\star| \leq 3^d 100 \mu_\beta K$. Now let $r(u) \in G^\prime$ be such that $r(u)\notin A^\star$. We will now argue that $D(V_\mz^k, V_u^k) \geq \delta k^\theta$, which shows the claim. \\
	
	Let $P$ be a path on $\Z^d$ between $V_\mz^k$ and $V_u^k$. Say that the path goes through the blocks $V_\mz^k, V_{a_1}^k, \ldots , V_{a_l}^k$, in this order, and listing blocks several times if the path leaves a block and returns again. Then $r(\mz),r(a_1),\ldots,r(a_l)=r(u)$ is a path in $G^\prime$. Let $r(\mz),r(u_1),\ldots,r(u_j)=r(u)$ be the loop-erasure of this path. 
	Then $\mz=u_0,u_1,\ldots,u_j=u$ is a sequence of vertices such that $V_{u_{i-1}}^k \sim V_{u_i}^k$ for $i=1,\ldots,j$. As $r(u) \notin A^\star$, there needs to exist at least one block $V_{u_i}^k$ with $i\in \{1,\ldots,j-1\}$ which is good. Let $i$ be the index of the first good block. 
	
	If $\|u_i-u_{i+1}\|_\infty \geq 2$, then the path $P$ needs to have a length of at least $\delta k^\theta$, as $V_{u_i}^k$ is a good block, and thus there are no $x,y\in V_{u_i}^k$ with $x\neq y$ and $x\sim \left(V_{u_i}^k\right)^c, y\sim \left(V_{u_i}^k\right)^c$.
	
	If $\|u_i-u_{i+1}\|_\infty = 1$, then the path $P$ crosses the annulus $\bigcup_{w:\|w-u_i\|_\infty = 1} V_w^k$ and thus has a length of at least $\delta k^\theta$, as $V_{u_i}^k$ is a good block. The path needs to cross the annulus as $V_{u_i}^k$ was assumed to be the first good block among $V_{u_1}^k,\ldots, V_{u_{j-1}}^k$, so in particular $r(u_{i}) \in A^\prime$ and thus $r(u_{i+1}) \in A^\star$, but $r(u) \notin A^\star$.
\end{proof}

\subsection{Proof of Theorem \ref{theo:main}}

With the previous results, we are ready to go to the proof of Theorem \ref{theo:main}.

\begin{proof}[Proof of Theorem \ref{theo:main}]
	The size of the ball $\cB_{\delta k^\theta}$ is less than $k^d$ times the number of blocks of the form $V_u^k$ that the ball $\cB_{\delta k^\theta}$ intersects.
	Using the definition of the $k$-box-count in Proposition \ref{propo:main}, we directly get that for all $k\in \N$
	\begin{equation*}
		|\cB_{\delta k^\theta}| \leq k^d \left|X_k\left(B_{\delta k^\theta} \left(V_\mz^k; D\right)\right)\right|
	\end{equation*}
	and using Proposition \ref{propo:main} we get for all $k,K$ large enough that
	\begin{equation*}
		\p_\beta \left( |\cB_{\delta k^\theta}| > k^d 3^d 100 \mu_\beta K \right)
		\leq 
		\p_\beta \left( X_k \left(B_{\delta k^\theta}(V_\mz^k; D)\right) > 3^d 100 \mu_\beta K \right) \leq 4 \cdot 2^{-K}.
	\end{equation*}
	For $n \in \N$, define $k_n\in \N$ by $k_n \coloneqq \lceil (n/\delta)^{1/\theta} \rceil$, so that $\delta k_n^\theta \approx n$.
	Then we get that for a large enough constant $C$ one has for all large enough $n$ and $K$ that
	\begin{equation*}
		\p_\beta \left( |\cB_n| > C n^{\frac{d}{\theta}} K \right)
		\leq
		\p_\beta \left( |\cB_{\delta k_n^\theta}| > k_n^d 3^d 100 \mu_{\beta} K \right)
		\leq 1.5^{-K}.
	\end{equation*}
\end{proof}

\subsection{Proof of Theorem \ref{theo:lower tail}}

Whereas Lemma \ref{lem:good sets} is a result about connected sets in the random graph, we need a similar result for paths instead of sets. The proof of it uses similar techniques as the proofs of \cite[Lemma 5.9]{baumler2022behavior} and \cite[Lemma 5.5]{baumler2022distances}. However, there is a stark difference, namely that we condition that certain blocks are connected. We use the same notation as in Section \ref{sec:previous}, i.e., we identify blocks of the form $V_u^k$ to a single vertex $r(u)$ and call the resulting graph $G^\prime = \left( V^\prime , E^\prime \right)$. Further, we assume that $k$ is large enough so that the results of Section \ref{sec:previous} (in particular, Lemmas \ref{lem:goodeventB}-\ref{lem:goodeventD}) hold.

\begin{lemma}\label{lem:good paths}
	Let $G^\prime = (V^\prime, E^\prime)$ be the graph that results from contracting boxes of the form $V_u^N$ into vertices $r(u)$.
	Let $r(\mz)=r(v_0),r(v_1),\ldots,r(v_K)$ be distinct vertices in $V^\prime$. Then for all large enough $K,N$ one has
	\begin{equation*}
		\p_\beta \left( \text{$r(v_{3^{d}}),\ldots,r(v_{K-3^d})$ are bad } \Big| r(v_0)\sim r(v_{1}) \sim \ldots \sim r(v_K) \right) \leq e^{-3 \mu_\beta K}.
	\end{equation*}
\end{lemma}

Before proving this lemma, we introduce the following claim.

\begin{claim}\label{claim:path degree}
	Let $r(\mz)=r(v_0),r(v_1),\ldots,r(v_K)$ be distinct vertices in $V^\prime$. Define the sets $Z=\left\{r(v_0),\ldots,r(v_K)\right\}$ and 
	\begin{equation*}
		Z^\cN = \bigcup_{r(u) \in Z} \cN\left(r(u)\right).
	\end{equation*}
	Then
	\begin{equation*}
		\p_\beta \left( \sum_{r(v) \in Z^{\cN} } \deg\left(r(v)\right) > 20 \mu_{\beta} \left|Z^{\cN}\right| \Big| r(v_0)\sim r(v_{1}) \sim \ldots \sim r(v_K) \right) \leq e^{-6 \mu_\beta K}.
	\end{equation*}
\end{claim}
\begin{proof}
	The random variable $\sum_{r(v) \in Z^{\cN} } \deg\left(r(v)\right)$ is {\sl not} the sum of independent Bernoulli random variables, as edges with both endpoints in $Z^{\cN}$ contribute twice to this sum if they are open. However, if we define the edge set
	\begin{equation*}
		F=\left\{\{r(u),r(v)\}\in E^\prime : r(u) \in Z^{\cN}\right\}
	\end{equation*}
	that consists of all edges with at least one endpoint in $Z^{\cN}$, we have that
	\begin{equation*}
		\sum_{r(v) \in Z^{\cN} } \deg\left(r(v)\right) \leq 2 \sum_{e \in F} \mathbbm{1}_{e \text{ open}}.
	\end{equation*}
	We split the set $F$ into two components $F=F_1 \cup F_2$ by defining $F_1 = \bigcup_{i=1}^{K} \left\{\{r(v_{i-1}),r(v_i)\}\right\}$ and $F_2 = F\setminus F_1$. So in particular, we get that $|F_1| = K$ and that the edges in $F_2$ are independent of the event $\{r(v_0)\sim \ldots \sim r(v_K)\}$. Using this independence implies that
	\begin{align}
		& \notag \p_\beta \left( \sum_{r(v) \in Z^{\cN} } \deg\left(r(v)\right) > 20 \mu_{\beta} \left|Z^{\cN}\right| \Big| r(v_0)\sim r(v_{1}) \sim \ldots \sim r(v_K) \right) \\
		& \notag
		\leq
		\p_\beta \left( \sum_{e \in F} \mathbbm{1}_{e \text{ open}} > 10 \mu_{\beta} \left|Z^{\cN}\right| \Big| r(v_0)\sim r(v_{1}) \sim \ldots \sim r(v_K) \right) \\
		& \notag
		\leq
		\p_\beta \left( \sum_{e \in F_2} \mathbbm{1}_{e \text{ open}} > 9 \mu_{\beta} \left|Z^{\cN}\right| \Big| r(v_0)\sim r(v_{1}) \sim \ldots \sim r(v_K) \right)\\
		& \label{eq:claim}
		= 
		\p_\beta \left( \sum_{e \in F_2} \mathbbm{1}_{e \text{ open}} > 9 \mu_{\beta} \left|Z^{\cN}\right| \right).
	\end{align}
	The set $F_2$ is a subset of the set of edges with at least one endpoint in $Z^{\cN}$. So in particular, we get that
	\begin{equation*}
		\sum_{e \in F_2} \p_\beta \left(e \text{ open}\right)
		\leq
		\sum_{e \in F} \p_\beta \left(e \text{ open}\right)
		\leq
		\mu_\beta |Z^{\cN}|.
	\end{equation*}
	The random variable $\sum_{e \in F_2} \mathbbm{1}_{e \text{ open}}$ is the sum of independent Bernoulli-distributed random variables. So we can use inequality \eqref{eq:exp moment bernoulli prob bound} to bound the probability that this random variable differs from its mean by a large factor. Thus we get that
	\begin{align*}
		&\p_\beta \left( \sum_{e \in F_2} \mathbbm{1}_{e \text{ open}} > 9 \mu_{\beta} \left|Z^{\cN}\right| \right)
		\overset{\eqref{eq:exp moment bernoulli prob bound}}{\leq}
		\exp\left(e \E\left[\sum_{e \in F_2} \mathbbm{1}_{e \text{ open}} \right] - 9 \mu_{\beta} \left|Z^{\cN}\right|  \right)\\
		&
		\leq
		\exp\left(e |Z^{\cN}| \mu_{\beta} - 9 \mu_{\beta} \left|Z^{\cN}\right|  \right)
		\leq e^{-6 \mu_\beta K}.
	\end{align*}
	Inserting this into \eqref{eq:claim} finishes the proof.
\end{proof}

Using Claim \ref{claim:path degree}, we can prove Lemma \ref{lem:good paths}.

\begin{proof}[Proof of Lemma \ref{lem:good paths}]
	Define the set $Z=\left\{r(v_0),r(v_1),\ldots,r(v_K) \right\}$.
	For such a set, we add the nearest neighbors to it. Formally, we define the set
	\begin{align*}
		Z^\cN = \bigcup_{ r(v) \in Z }  \cN\left(r(v)\right)
	\end{align*}
	which is still a connected set and satisfies $K \leq |Z^\cN| \leq 3^d K$. A vertex $r(u) \in G^\prime$ can be included into the set $Z^{\cN}$ in more than one way, meaning that there can be different vertices $r(v),r(\tilde{v}) \in Z$ such that $r(u) \in \cN\left(r(v)\right)$ and $r(u) \in \cN\left(r(\tilde{v})\right)$.
	However, each vertex $r(u) \in G^\prime$ can be included into the set $Z^{\cN}$ 
	in at most $3^d$ many different ways. So in particular, we have
	\begin{align*}
		\sum_{i=0}^K \deg^{\cN} \left(r(v_i)\right) \leq 3^d \sum_{r(v) \in Z^{\cN} } \deg\left(r(v)\right) .
	\end{align*}
	Next, we iteratively define a set $\mathbb{LI} = \mathbb{LI}(Z) = \mathbb{LI}_{K-3^d-1} \subset Z$ as follows:
	\begin{enumerate}\addtocounter{enumi}{-1}
		\item Start with $\mathbb{LI}_{3^d} = \emptyset$.
		\item For $i=3^d+1,\ldots, K-3^d-1$: If 
		\begin{equation*}
			\text{$\deg^{\cN} \left(r(v_i)\right) \leq 9^d 50 \mu_{\beta}$ and $ \cN \left( r(v_i) \right)  \nsim \bigcup_{r(u) \in \mathbb{LI}_{i-1}} \cN(r(u))$,}
		\end{equation*}
		then set $\mathbb{LI}_{i} = \mathbb{LI}_{i-1} \cup \{r(v_i)\}$; else set $\mathbb{LI}_{i}= \mathbb{LI}_{i-1}$.
	\end{enumerate}
	On the event where $\overline{\deg}(Z^\cN) \leq 20 \mu_{\beta}$ (which is very likely by Claim \ref{claim:path degree}), we have
	\begin{align*}
		\sum_{i=0}^K \deg^{\cN} \left(r(v_i)\right)  \leq 3^d \sum_{r(v) \in Z^{\cN} } \deg\left(r(v)\right) \leq 3^d 20 \mu_{\beta} \left|Z^{\cN}\right| \leq 9^d 20 \mu_{\beta} K
	\end{align*}
	and thus there can be at most $\frac{K}{2}$ many vertices $r(v_i) \in Z$ with $\deg^{\cN} \left(r(v_i)\right) > 9^d 50 \mu_{\beta}$, which implies that there are at least  $\frac{K}{2}$ many vertices with $\deg^{\cN} \left(r(v_i)\right) \leq 9^d 50 \mu_\beta$. Whenever we include such a vertex in the set $\mathbb{LI}$, it can block at most $3^d \deg^{\cN} \left(r(v_i)\right) \leq 27^d 50\mu_{\beta}$ different vertices from entering the set $\mathbb{LI}$, which already implies
	\begin{equation*}
		|\mathbb{LI}| \geq \frac{K-2 \cdot 3^d -2}{2 (27^d 50 \mu_{\beta} + 1)} \geq \frac{K}{27^d 200 \mu_{\beta}}  
	\end{equation*}
	for $K$ large enough.

	Conditioned on the (neighborhood) degree of the vertex $r(w)$, and assuming that $\deg^{\cN} \left( r(w)\right) \leq 9^d 50 \mu_{\beta}$, the probability that $r(w)$ is not $\delta$-good is bounded by
	\begin{align*}
		\deg\left(r(w)\right)^2 \left(1-f_2(\delta)\right) +
		\deg^{\cN}\left( r(w) \right) \left(1-f_1(\delta)\right)
		+ (1-f_3(\delta))
		\leq 9^{2d} 5000 \mu_{\beta}^2 \left(1-f(\delta)\right) \text ,
	\end{align*}
	where $f$ was defined by $f(\delta)= \min\{f_1(\delta) , f_2(\delta) , f_3(\delta)\}$.
	Remember that we chose $\delta>0$ in \eqref{eq:deltasmall2} small enough so that
	\begin{align}\label{eq:deltasmall3}
		\left(9^{2d} 5000 \mu_{\beta}^2 \left(1-f(\delta)\right)\right)^{\frac{1}{27^d200\mu_{\beta}}} \leq e^{-5\mu_\beta } .
	\end{align}
	We now claim that the set $\mathbb{LI}$ contains at least one good vertex with high probability. Given the graph $G^\prime$, it is independent whether different vertices in $\mathbb{LI}$ are good or not, as we will argue now. For a vertex $r(u)$, it depends only on edges with at least one end in the set $\bigcup_{r(v) \in \cN\left(r(u)\right)} V_v^{N}$ whether the vertex $r(u)$ is good or not. But for different vertices $r(u), r(u^\prime) \in \mathbb{LI}$ there are no edges with one end in $\bigcup_{r(v) \in \cN\left(r(u)\right)} V_v^{N}$ and the other end in $\bigcup_{r(v) \in \cN\left(r(u^\prime)\right)} V_v^{N}$, as $\cN\left(r(u)\right) \nsim \cN\left(r(u^\prime)\right)$. Thus, it is independent whether different vertices in $\mathbb{LI}$ are good. So in particular, the probability that there is no good vertex in the set $\mathbb{LI}$ is bounded by
	\begin{align}\label{eq:tilde f}
		\left( 9^{2d} 5000 \mu_{\beta}^2 (1-f(\delta))  \right)^{|\mathbb{LI}|}
		\leq  \left( 9^{2d} 5000 \mu_{\beta}^2 (1-f(\delta))  \right)^{\frac{K}{27^d200\mu_{\beta}}} \overset{\eqref{eq:deltasmall3}}{\leq} e^{-5\mu_\beta K}
	\end{align}
	and thus the set $\mathbb{LI}$ $\big($and also the set $\{r(v_{3^d+1}),\ldots, r(v_{K-3^d-1})\}\big)$ contains at least one good vertex with very high probability. So in total, we see that
	\begin{align*}
		&  \p_\beta \left( \text{$r(v_{3^{d}}),\ldots,r(v_{K-3^d})$ are bad } \Big|  r(v_0)\sim r(v_{1}) \sim \ldots \sim r(v_K) \right)\\
		&
		\leq 
		\p_\beta \left( \text{$r(v_{3^{d}}),\ldots,r(v_{K-3^d})$ are bad } \Big| \overline{\deg}(Z^{\cN}) \leq 20 \mu_{\beta},  r(v_0) \sim \ldots \sim r(v_K) \right)\\
		&
		+
		\p_\beta \left(\overline{\deg}(Z^{\cN}) > 20 \mu_{\beta} \Big| r(v_0)\sim r(v_{1}) \sim \ldots \sim r(v_K) \right)
		\leq e^{-5 \mu_\beta K} + e^{-6 \mu_\beta K} \leq e^{-3 \mu_\beta K}
	\end{align*}
	for all $K$ large enough; the first inequality in the above line of inequalities follows from the elementary inequality $\p(A|B) \leq \p(A|B,C) + \p(C^c|B)$ that holds for all probability measures and events $A,B,C$, and the second inequality follows from \eqref{eq:tilde f} and Claim \ref{claim:path degree}.
\end{proof}

One other intermediary result that we need is that there are no points that are `exceptionally likely' to connect to the origin. For $x\in \Z^d \setminus \{\mz\}$ we write $S_{\mz,x}^k$ for the set of $(k+1)$-tuples  $\mz=v_0,v_1,\ldots,v_k=x$ that start at $\mz$, end at $x$, and that are self-avoiding, i.e., $v_i \neq v_j$ for $i\neq j$.
For such a sequence $v_0,v_1,\ldots,v_k$ we also use the notation $\overline{v}$, and we say that the path $\overline{v} = (v_0,\ldots,v_k)$ is open if all edges $\{v_{i-1},v_i\}, i=1,\ldots,k$ are open. 

\begin{lemma}\label{lem:poly bound}
	There exists a constant $C=C(d,\beta)$ such that for all $x\in \Z^d$
	\begin{equation*}
		\sum_{\overline{v}\in S_{\mz,x}^k} \p_\beta \left( \overline{v} \text{ open}\right)
		\leq C \mu_\beta^{3k} \frac{1}{\|x\|^{2d}}.
	\end{equation*}
\end{lemma}

\begin{proof}
	If the path $\overline{v}=(\mz=v_0,\ldots,v_k=x)$ is of length $k$ from $\mz$ to $x$, there needs to exist $j\in \{1,\ldots,k\}$ such that $\|v_{j-1}-v_j\|\geq \frac{\|x\|}{k}$. Let $S_{\mz,x}^{k,j}$ be the set of elements  $(v_0,\ldots,v_k) \in S_{\mz,x}^k$ for which $\|v_{j-1}-v_j\|\geq \frac{\|x\|}{k}$. Thus we get that
	\begin{align}\label{eq:insert union bound}
		\sum_{\overline{v}\in S_{\mz,x}^k} \p_\beta \left( \overline{v} \text{ open}\right)
		\leq
		\sum_{j=1}^{k} \sum_{(v_0,\ldots,v_k)\in S_{\mz,x}^{k,j}} \ \prod_{i=1}^{k} \p_\beta \left( v_{i-1} \sim v_i\right).
	\end{align}
	Let $C_1 < \infty$ be a constant so that $\p_\beta(a\sim b) \leq C_1 \|a-b\|^{-2d}$ for all $a,b\in \Z^d$.
	For each of the $k$ many summands in the above sum we then have
	\begin{align*}
		& \sum_{(v_0,\ldots,v_k)\in S_{\mz,x}^{k,j}} \ \prod_{i=1}^{k} \p_\beta \left( v_{i-1} \sim v_i\right)
		\leq
		C_1 \left(\|x\|/k\right)^{-2d} 
		\sum_{(v_0,\ldots,v_k)\in S_{\mz,x}^{k,j}} \ \prod_{\substack{i=1:\\ i \neq j}}^{k} \p_\beta \left( v_{i-1} \sim v_i\right)
		\\
		& 
		\leq
		C_1 \left(\|x\|/k\right)^{-2d} 
		\sum_{\mz=v_0,v_1,\ldots,v_k=x} \ \prod_{\substack{i=1:\\i\neq j}}^{k} \p_\beta \left( v_{i-1} \sim v_i\right)
		\\
		&
		=
		C_1 \left(\|x\|/k\right)^{-2d} 
		\sum_{\mz=v_0,v_1,\ldots,v_{j-1}} \ \prod_{i=1}^{j-1} \p_\beta \left( v_{i-1} \sim v_i\right)
		\sum_{v_j,\ldots,v_{k}=x} \ \prod_{i=j+1}^{k} \p_\beta \left( v_{i-1} \sim v_i\right)
		\\
		&
		=
		C_1 \left(\|x\|/k\right)^{-2d} \mu_\beta^{j-1} \mu_\beta^{k-j}
		\leq
		C_1 \|x\|^{-2d} k^{2d} \mu_\beta^{k} 
	\end{align*}
	where the sums of the form $\sum_{\mz=v_0,v_1,\ldots,v_{\ell}}$ range over distinct values of $v_0,\ldots,v_{\ell}$. Inserting this into \eqref{eq:insert union bound} shows that
	\begin{align*}
		\sum_{\overline{v}\in S_{\mz,x}^k} \p_\beta \left( \overline{v} \text{ open}\right) \leq
		C_1 \|x\|^{-2d} k^{2d+1} \mu_\beta^{k}
		\leq
		C_1 \|x\|^{-2d} k^{3d} \mu_\beta^{k}
		\leq
		C_2 \|x\|^{-2d} \mu_\beta^{3k}
	\end{align*}
	for some constant $C_2 < \infty$. The last inequality follows from elementary analysis using that $\mu_\beta \geq 2$, as the nearest-neighbor edges in the $\infty$-norm are always open.
\end{proof}

\begin{remark}\label{remark:x to inf}
	If $k\in \N$ is fixed, then Lemma \ref{lem:poly bound} directly implies that $\p_\beta \left( D(\mz,x) \leq k \right) = \mathcal{O}(\|x\|^{-2d})$ as $x\to \infty$.
\end{remark}

With this, we are finally ready to prove Theorem \ref{theo:lower tail}. The proof of the theorem is split up into two different parts, depending on which of the two inequalities in \eqref{eq:lower tail} we want to prove. In the first part, we prove that for all $d$ and $\beta> 0$, there exists a constant $c>0$ such that for all $x\in \Z^d\setminus\{\mz\}$ and all $\eps \in (0,1)$ with $\eps \|x\|^{\theta(\beta,d)} \geq 1$ one has
\begin{equation}\label{eq:lower tail 1}
	c \eps^{\frac{2d}{\theta}} \leq
	\p_\beta \left( \frac{D(\mz,x)}{\|x\|^{\theta}} \leq \eps  \right) .
\end{equation}
In the second part, we prove that for all $d$ and $\beta> 0$ there exists a constant $C<\infty$ such that for all $x\in \Z^d\setminus\{\mz\}$ and all $\eps \in (0,1)$ with $\eps \|x\|^{\theta(\beta,d)} \geq 1$, one has
\begin{equation}\label{eq:lower tail 2}
	\p_\beta \left( \frac{D(\mz,x)}{\|x\|^{\theta}} \leq \eps  \right) 
	\leq C \eps^{\frac{2d}{\theta}}.
\end{equation}
Together, this implies that Theorem \ref{theo:lower tail} holds.

\begin{proof}[Proof of \eqref{eq:lower tail 1}]
	Using that $\dia \left(B_n(\mz)\right) \approx_P n^\theta$, one gets that there exists a constant $\bar{c}=\bar{c}(\beta,d)$ such that for all $x\in \Z^d \setminus \mz$
	\begin{equation*}
		\p_\beta \left(\dia \left(B_{\bar{c}\eps^{1/\theta} \|x\| } (\mz)\right) \leq \lfloor \tfrac{1}{3} \eps \|x\|^\theta \rfloor  \right) > 0.5 .
	\end{equation*}
	We now consider the event $\cA$ defined by
	\begin{align*}
		\cA 
		=&
		\left\{\dia \left(B_{\bar{c}\eps^{1/\theta} \|x\| } (\mz)\right) \leq \lfloor \tfrac{1}{3} \eps \|x\|^\theta \rfloor  \right\} 
		\cap 
		\left\{\dia \left(B_{\bar{c}\eps^{1/\theta} \|x\| } (x)\right) \leq \lfloor \tfrac{1}{3} \eps \|x\|^\theta \rfloor \right\}\\
		&
		\cap 
		\left\{ B_{\bar{c}\eps^{1/\theta} \|x\| } (\mz) \sim B_{\bar{c}\eps^{1/\theta} \|x\| } (x) \right\} .
	\end{align*}
	As all three events in the definition of $\cA$ are increasing, we get by the FKG-inequality \cite[Section 1.3]{heydenreich2017progress} that
	\begin{align}\label{eq:one minus exp}
		\notag \p_\beta \left(\cA\right) & \geq \frac{1}{4} \p_\beta \left( B_{\bar{c}\eps^{1/\theta} \|x\|}  (\mz) \sim B_{\bar{c}\eps^{1/\theta} \|x\| }  (x) \right)\\
		&
		= \frac{1}{4} \left(1 - \exp \left( - \beta \sum_{a \in  B_{\bar{c}\eps^{1/\theta} \|x\|}  (\mz)} \sum_{b \in  B_{\bar{c}\eps^{1/\theta} \|x\|}  (x)} J(a-b) \right)\right)
	\end{align}
	The two sets $B_{\bar{c}\eps^{1/\theta} \|x\|}  (\mz), B_{\bar{c}\eps^{1/\theta} \|x\| }  (x)$ have a size of order $\eps^{d/\theta} \|x\|^d$ each, and points $a \in B_{\bar{c}\eps^{1/\theta} \|x\|}  (\mz)$ and $b \in B_{\bar{c}\eps^{1/\theta} \|x\|}  (x)$ have a Euclidean distance of order  $\|x\|$. So in particular, there exists a constant $c^\prime$ such that
	\begin{equation*}
		\sum_{a \in  B_{\bar{c}\eps^{1/\theta} \|x\|}  (\mz)} \ \sum_{b \in  B_{\bar{c}\eps^{1/\theta} \|x\|}  (x)} J(a-b) 
		\geq
		c^\prime \eps^{d/\theta} \|x\|^d \eps^{d/\theta} \|x\|^d  \frac{1}{\|x\|^{2d}} =  c^\prime \eps^{2d/\theta}.
	\end{equation*}
	Inserting this into $\eqref{eq:one minus exp}$ shows that $\p_\beta \left(\cA\right) \geq c \eps^{2d/\theta}$ for some small enough $c>0$. Whenever the event $\cA$ holds one can find a path between $\mz$ and $x$ that uses an open edge $\{a,b\}$ with $a \in B_{\bar{c}\eps^{1/\theta} \|x\|}  (\mz), b \in B_{\bar{c}\eps^{1/\theta} \|x\|} (x)$. In particular, one gets that
	\begin{equation*}
		D(\mz,x) \leq \dia \left(B_{\bar{c}\eps^{1/\theta} \|x\| } (\mz)\right) + 1 + \dia \left(B_{\bar{c}\eps^{1/\theta} \|x\| } (x)\right)
		\leq 2 \lfloor \tfrac{1}{3} \eps \|x\|^\theta \rfloor + 1
		\leq  \eps \|x\|^\theta .
	\end{equation*}
	Here, the second inequality follows from the definition of $\cA$ and the final inequality follows from the assumption $\eps \|x\|^\theta \geq 1$ $\big($note that the final inequality holds both for $\eps \|x\|^\theta \in \left[1,3\right)$ and for $\eps \|x\|^\theta \in \left[3,\infty\right) \big)$. So in total, we see that
	\begin{equation*}
		\p_\beta \left( \frac{D(\mz,x)}{\|x\|^{\theta}} \leq \eps  \right) 
		\geq
		\p_\beta \left( \cA \right) \geq c \eps^{\frac{2d}{\theta}}
	\end{equation*}
	which finishes the proof of \eqref{eq:lower tail 1}.
\end{proof}

\begin{proof}[Proof of \eqref{eq:lower tail 2}]
	The proof uses a renormalization argument. For $\eps \in(0,1)$ and $x\in \Z^d$ with $\eps \|x\|^\theta \geq 1$ we define
	\begin{equation*}
		M=M(\eps, x) \coloneqq \Big \lceil (\eps/\delta)^{\frac{1}{\theta}} \|x\| \Big \rceil
	\end{equation*}
	where $\delta > 0$ is the parameter chosen in \eqref{eq:deltasmall}, \eqref{eq:deltasmall2}.
	We define the renormalized graph $G^\prime = (V^\prime, E^\prime)$ by contracting the boxes $\left(V_u^M\right)_{u\in \Z^d}$ into vertices $\left(r(u)\right)_{u\in \Z^d}$. Let $K^\prime \geq 2\cdot 3^d$ and $N^\prime$ be large enough so that the result of Lemma \ref{lem:good paths} holds for all $K \geq K^\prime, N \geq N^\prime$.
	
	If $M\leq N^\prime$, then $\eps \|x\|^\theta = \Theta(1)$ and thus there exist constants $C, C^\prime$ such that
	\begin{equation*}
		\p_\beta \left(D(\mz,x) \leq \eps \|x\|^\theta\right) 
		\leq C \|x\|^{-2d} \leq C^\prime \eps^{\frac{2d}{\theta}},
	\end{equation*}
	where the first inequality holds by Remark \ref{remark:x to inf}.
	
	Thus, let us assume from here on that $M \geq N^\prime$.
	Assume that a path $P\subset \Z^d$ crosses a good block $V_u^M$ in the sense that it starts outside of $\bigcup_{r(v) \in \cN(r(u))} V_v^M$, then touches $V_u^M$, and then leaves the set $\bigcup_{r(v) \in \cN(r(u))} V_v^M$. As the block $V_u^M$ was assumed to be good, the path $P$ needs to have a length of at least $\delta M^\theta \geq \eps \|x\|^\theta$. For $x\in \Z^d$, let $\overline{x}$ be the unique value $\overline{x} \in \Z^d$ for which $x \in V_{\overline{x}}^M$. Thus we see that the event $\left\{ D(\mz,x) \leq \eps \|x\|^\theta \right\}$ implies (at least) one of the two following events holds.
	\begin{align}
		\label{eq:event of interest1} & D_{G^\prime}(r(\mz),r(\overline{x})) \leq K^\prime, \ \ \ \ \ \ \text{ or there exits $K\geq K^\prime$ such that}\\
		&
		\label{eq:event of interest2} \text{$\exists P = \left(r(\mz),r(v_1), \ldots, r(v_{K}) = r(\overline{x})\right)$ open in $G^\prime$ s.t. $r(v_{3^d}),\ldots,r(v_{K-3^d})$ are bad}.
	\end{align}
	By the definition $\overline{x}$ we have that $\|\overline{x}\| \approx \eps^{-1/\theta}$ and thus
	\begin{equation}\label{eq:event of interest1 bound}
		\p_\beta \left( D_{G^\prime}(r(\mz),r(\overline{x})) \leq K^\prime \right) 
		= 
		\p_\beta \left( D (\mz,\overline{x}) \leq K^\prime \right) = \mathcal{O}\left(\|\overline{x}\|^{-2d}\right) = \mathcal{O}\left(\eps^{2d/\theta}\right)
	\end{equation}
	by Remark \ref{remark:x to inf}. In order to estimate the probability of the event in \eqref{eq:event of interest2} we use a union bound over the set of possible choices of vertices $v_1,\ldots,v_{K-1},\overline{x}$ and Lemma \ref{lem:good paths}. In analogy to Lemma \ref{lem:poly bound}, we write $S_{\mz,\overline{x}}^{\prime K}$ for the set of possible paths $r(v_0)=r(\mz), r(v_1),\ldots, r(v_{K-1})$, $ r(v_K) = r(\overline{x})$ from $r(\mz)$ to $r(\overline{x})$ of $K$ steps in the graph $G^\prime$. We write $\widetilde{r(v)}$ for a sequence of distinct vertices $\left(r(v_0)=r(\mz), r(v_1),\ldots, r(v_{K-1}), r(v_K) = r(\overline{x})\right)$, and we say that $\widetilde{r(v)}$ is open if $r(v_{i-1}) \sim r(v_i)$ for $i=1,\ldots,K$. Using this notation we get that
	\begin{align*}
		& \sum_{\widetilde{r(v)}\in S_{\mz,\overline{x}}^{\prime K} }
		\p_\beta \left( r(v_{3d}) ,\ldots, r(v_{K-3^d}) \text{ bad and } \widetilde{r(v)} \text{ open}   \right) \\
		&
		= \sum_{\widetilde{r(v)}\in S_{\mz,\overline{x}}^{\prime K} }
		\p_\beta \left( r(v_{3d}) ,\ldots, r(v_{K-3^d}) \text{ bad} \ \big| \  \widetilde{r(v)} \text{ open}   \right) \p_\beta \left(\widetilde{r(v)} \text{ open}\right)\\
		&
		\leq \sum_{\widetilde{r(v)}\in S_{\mz,\overline{x}}^{\prime K} }
		e^{-3\mu_\beta K} \p_\beta \left(\widetilde{r(v)} \text{ open}\right)
		\leq 
		\frac{C}{\|\overline{x}\|^{2d}} \mu_\beta^{3K}
		e^{-3\mu_\beta K} \leq \frac{C}{\|\overline{x}\|^{2d}} 2^{-K}
	\end{align*}
	where the first inequality follows from Lemma \ref{lem:good paths} and the second inequality follows from Lemma \ref{lem:poly bound}. Summing this over all possible values $K\geq K^\prime$ shows that the event in \eqref{eq:event of interest2} has a probability of at most $\tfrac{2C}{\|\overline{x}\|^{2d}}$. By the definition of the events in \eqref{eq:event of interest1} and \eqref{eq:event of interest2} we thus get that
	\begin{align*}
		& \p_\beta \left( D(\mz,x) \leq \eps \|x\|^\theta \right) \leq \p_\beta \left( D_{G^\prime}(r(\mz),r(\overline{x})) \leq K^\prime \right) + \sum_{K=K^\prime}^{\infty} \frac{C}{\|\overline{x}\|^{2d}} 2^{-K} \\
		&
		=
		\mathcal{O}\left(\|\overline{x}\|^{-2d}\right) = \mathcal{O}\left(\eps^{\frac{2d}{\theta}}\right)
	\end{align*}
	which finishes the proof.
\end{proof}

\subsection{Proof of Theorem \ref{theo:hausdorff dimension}}

Finally, we can prove Theorem \ref{theo:hausdorff dimension}. Our proof uses an argument that estimates the maximal "distortion" of the metric space $(X,\hat{D})$ compared to the metric space $(X,\|\cdot \|)$. A similar argument was used for the upper bound on the Hausdorff dimension of the metric space in \cite[Theorem 1.16]{ding2023uniqueness}. The difference is that we use Proposition \ref{propo:main} as an input, whereas the proof in \cite{ding2023uniqueness} takes the stretched-exponential moment bound \eqref{eq:stretched upper bound} as an input. These different inputs work in opposite ways: whereas \eqref{eq:stretched upper bound} allows to prove upper bounds on $\dim_{\cH}^\beta(X)$, Proposition \ref{propo:main} allows to prove lower bounds on $\dim_{\cH}^\beta(X)$.

\begin{proof}[Proof of Theorem \ref{theo:hausdorff dimension}]
	By the countable stability of the Hausdorff dimension we can assume that $X$ is contained in a compact subset of $\R^d$, and by the translation invariance we can assume that $X\subseteq \left[0,1\right]^d$. For two sets $A,B\subseteq \left[0,1\right]^d$, define $\hat{D}(A,B)=\inf \left\{\hat{D}(x,y):x\in A, y\in B\right\}$.
	For $u \in \{0\ldots,m\}^d$, define
	\begin{equation*}
		Q_m^u = \left|\left\{v \in \Z^d : \hat{D}\left( \frac{1}{m} u + \left[0,\frac{1}{m} \right)^d, \frac{1}{m} v + \left[0,\frac{1}{m} \right)^d \right) < m^{-\theta} \right\} \right| .
	\end{equation*}
	We are interested in the event $\left\{Q_m^u > M \log (m) \right\}$ for $m=2^l$ for some $l\in \N$. For $u\in \Z^d$ and $m\in \N$, define the box $W_u^{1/m}\subset \R^d$ by
	\begin{equation*}
		W_u^{1/m} = \frac{1}{m} u + \left[0,\frac{1}{m} \right)^d.
	\end{equation*}
	As $\hat{D}_n(\cdot, \cdot) = \hat{a}_n^{-1} D \left( \lfloor n \cdot \rfloor, \lfloor n \cdot \rfloor \right)$ converges to $\hat{D}$  in law with respect to the topology
	of local uniform convergence on $\R^d \times \R^d$, we get that
	\begin{align*}
		& \p \left(  Q_m^u > M \log (m) \right) \\
		&
		= 
		\lim_{n\to \infty} \p_\beta \left(\left|\left\{v \in \Z^d : \hat{D}_{2^n} \left( W_u^{1/m}, W_v^{1/m} \right) < m^{-\theta} \right\} \right| > M \log(m)\right)\\
		&
		= 
		\lim_{n\to \infty} \p_\beta \left(\left|\left\{v \in \Z^d : \hat{a}_{2^n}^{-1} D \left( V_u^{m^{-1}2^n}, V_v^{m^{-1}2^n} \right) < m^{-\theta} \right\} \right| > M \log(m)\right)\\
		&
		= 
		\lim_{n\to \infty} \p_\beta \left(\left|\left\{v \in \Z^d :  D \left( V_u^{m^{-1}2^n}, V_v^{m^{-1}2^n} \right) < \hat{a}_{2^n} m^{-\theta} \right\} \right| > M \log(m)\right). 
	\end{align*}
	Now we use that $a_{2^n} = \Theta \left(2^{\theta n}\right)$ and Proposition \ref{propo:main} to get that there exists a constant $c>0$ such that
	\begin{align*}
		& \p \left(  Q_m^u > M \log (m) \right) \\
		& 
		\leq
		\lim_{n\to \infty} \p_\beta \left(\left|\left\{v \in \Z^d :  D \left( V_u^{m^{-1}2^n}, V_v^{m^{-1}2^n} \right) < c \left(m^{-1}2^n\right)^\theta \right\} \right| > M \log(m)\right)\\
		&
		\leq e^{-c^\prime M \log(m)} \leq m^{-3d}
	\end{align*}
	for some $c^\prime >0$ and all $M$ large enough. Using a union bound over all possible values of $u \in \{0,\ldots, m\}^d$ shows that for $M,m$ large enough 
	\begin{equation*}
		\p \left( \exists u \in \{0,\ldots,m\}^d:  Q_m^u > M \log (m) \right) \leq m^{-d}
	\end{equation*}
	and using a Borel-Cantelli argument shows that there exists almost surely some $h \in \N$ such that for all $l\geq h$, and all $u\in \{0,\ldots,2^l\}^d$ one has $Q_{2^l}^u \leq M l$. Let $h\in \N$ be this finite number for the rest of the proof.
	
	With this result, we are ready to compare the distortions between the metric spaces $\left(X,\hat{D}\right)$ and $\left(X,\|\cdot\|_\infty\right)$. Fix an integer $M$ so that the previous results hold. Let $y \in \left[0,1\right]^d$ be a point, let $r \leq \frac{1}{2} 2^{-l\theta}$ with $l\geq h$, and say that  $y \in W_v^{2^{-l}}$ with $v\in \{0,\ldots,2^l\}^d$. Then
	\begin{align*}
		B_r(y;\hat{D})  \subseteq B_{\frac{1}{2} 2^{-l\theta}} (W_v^{2^{-l}}, \hat{D}) \subseteq \bigcup_{i=1}^{M l} W_{u_i}^{2^{-l}}
	\end{align*}
	for some $u_1,\ldots, u_{Ml} \in \Z^d$.
	So we see that we can cover all sets of the form $B_r(y,\hat{D})$ with $r \leq \frac{1}{2} 2^{-l\theta}$ and $y\in \left[0,1\right)^d$ with at most $Ml$ many sets of the form $W_{u_i}^{2^{-l}}$. Now let $X\subset \left[0,1\right]^d$ be a set and assume that $\dim_{\cH}^{\beta}(X) < \eta$. Then we get that
	\begin{equation*}
		\inf \left\{ \sum_{i=1}^{\infty} r_i^\eta \ \Big| \ X \subseteq \bigcup_{i=1}^{\infty} B_{r_i}\left(y_i;\hat{D}\right), y_1,y_2,\ldots \in \left[0,1\right]^d \right\}	 = 0 .
	\end{equation*}
	Assume that $\left(B_{r_i}\left(y_i;\hat{D}\right)\right)_{i\in \N}$ is a covering of $X$ with $r_i \leq 2^{-h\theta-1}$ for all $i \in \N$. For each set $B_{r_i}\left(y_i;\hat{D}\right)$ with $r_i \in \left( 2^{-l\theta-2}, 2^{-l\theta-1} \right]$ there exists a covering of $B_{r_i}\left(y_i;\hat{D}\right)$ with at most $Ml$ many sets of the form $W_u^{2^{-l}}$, i.e.,
	\begin{equation*}
		B_{r_i}\left(y_i;\hat{D}\right) \subseteq \bigcup_{j=1}^{Ml} W_{u_j^i}^{2^{-l}}.
	\end{equation*}
	Write $L_l \subset \N$ for the set of indices $i$ with $r_i \in \left( 2^{-l\theta-2}, 2^{-l\theta-1} \right]$. Thus we get that
	\begin{align*}
		X\subseteq \bigcup_{i=1}^{\infty} B_{r_i}\left(y_i;\hat{D}\right) = \bigcup_{l=h}^{\infty} \bigcup_{i\in L_l} B_{r_i}\left(y_i;\hat{D}\right)
		\subseteq
		\bigcup_{l=h}^{\infty} \bigcup_{i\in L_l} \bigcup_{j=1}^{Ml} W_{u_j^i}^{2^{-l}}.
	\end{align*}
	Each set of the form $W_{u_j^i}^{2^{-l}}$ is covered by a ball of radius $2^{-l}$ in the $\|\cdot\|_\infty$-norm centered at $u_j^i$ and thus we get that
	\begin{align*}
		X\subseteq \bigcup_{i=1}^{\infty} B_{r_i}\left(y_i;\hat{D}\right) = \bigcup_{l=h}^{\infty} \bigcup_{i\in L_l} B_{r_i}\left(y_i;\hat{D}\right)
		\subseteq
		\bigcup_{l=h}^{\infty} \bigcup_{i\in L_l} \bigcup_{j=1}^{Ml} B_{2^{-l}} \left(u_j^i;\|\cdot\|_\infty\right)
	\end{align*}
	and thus the collection of balls $\left\{ B_{2^{-l}} \left(u_j^i;\|\cdot\|_\infty\right) : l\geq h, i \in L_l , j \in \{1,\ldots,Ml\} \right\}$ is a cover of $X$.
	For a ball  $B_{r}\left(x;\|\cdot\|_\infty\right)$, we define 
	\begin{equation*}
		\rad\left( B_{r} \left(x ; \|\cdot\|_\infty \right) \right) \coloneqq r
	\end{equation*}
	as the radius of the ball. 
	Let $\gamma = \theta \eta$ and let $\eps > 0$. Then
	\begin{align*}
		&\sum_{l=h}^{\infty} \sum_{i\in L_l} \sum_{j=1}^{Ml} \rad \left( B_{2^{-l}} \left(u_j^i;\|\cdot\|_\infty\right) \right)^{\gamma + \eps}
		=
		\sum_{l=h}^{\infty} \sum_{i\in L_l} Ml 2^{-l\gamma} 2^{-l\eps} \\
		&
		\leq
		\sup_{l \geq h} \left\{Ml 2^{-l\eps} \right\}
		\sum_{l=h}^{\infty} \sum_{i\in L_l}  2^{-l\theta \eta}
		\leq
		4^\eta \sup_{l \geq h} \left\{Ml 2^{-l\eps} \right\}
		\sum_{l=h}^{\infty} \sum_{i\in L_l}  r_i^\eta\\
		&
		\leq
		4^\eta \sup_{l \geq h} \left\{Ml 2^{-l\eps}\right\} \sum_{i=1}^{\infty}  r_i^\eta.
	\end{align*}
	The infimum over all possible covers $\left(B_{r_i}\left(y_i;\hat{D}\right)\right)_{i\in \N}$ of $X$ in the last line of the above expression is $0$, as $\eta > \dim_{\cH}^\beta(X)$. Thus also the left-hand side can be arbitrarily small, which implies that $\dim_{\cH}^{0}(X) \leq \gamma + \eps$. As $\eps > 0$ was arbitrary, this implies that
	\begin{equation*}
		\dim_{\cH}^{0}(X) \leq \gamma = \theta \eta,
	\end{equation*}
	and as $\eta > \dim_{\cH}^\beta(X)$ was arbitrary, we finally get that
	\begin{equation*}
		\dim_{\cH}^{0}(X) \leq \theta \dim_{\cH}^{\beta}(X) .
	\end{equation*}
\end{proof}

\textbf{Acknowledgments.}
	The author would like to thank the anonymous referee for their helpful comments and corrections.

\end{document}